\documentclass[12pt,reqno]{amsart}
\usepackage{amsmath}
\usepackage{amssymb}
\usepackage{amstext}
\usepackage{mathrsfs}
\usepackage{a4wide}
\usepackage{graphicx}
\usepackage{bbm}
\usepackage{verbatim}
\usepackage{bm}

\allowdisplaybreaks \numberwithin{equation}{section}
\usepackage{color}
\usepackage{cases}

\numberwithin{equation}{section}

\newtheorem{theorem}{Theorem}[section]

\newtheorem{lemma}[theorem]{Lemma}

\theoremstyle{definition}

\theoremstyle{remark}
\newtheorem{remark}[theorem]{Remark}

\begin{document}

	\title
	[Vortex sheets for the 2D incompressible Euler equation]{Co-rotating and traveling vortex sheets for the 2D incompressible Euler equation} %\uppercase\expandafter{\romannumeral 1}: patch cases}
	
	\author{Daomin Cao, Guolin Qin, Changjun Zou}
	
	\address{Institute of Applied Mathematics, Chinese Academy of Sciences, Beijing 100190, and University of Chinese Academy of Sciences, Beijing 100049,  P.R. China}
	\email{dmcao@amt.ac.cn}
	\address{Institute of Applied Mathematics, Chinese Academy of Sciences, Beijing 100190, and University of Chinese Academy of Sciences, Beijing 100049,  P.R. China}
	\email{qinguolin18@mails.ucas.edu.cn}
	\address{Institute of Applied Mathematics, Chinese Academy of Sciences, Beijing 100190, and University of Chinese Academy of Sciences, Beijing 100049,  P.R. China}
	\email{zouchangjun17@mails.ucas.ac.cn}
	
	\thanks{This work was supported by NNSF of China Grant 11831009 and Chinese Academy of Sciences (No. QYZDJ-SSW-SYS021).}

	\begin{abstract}
		We construct co-rotating and traveling vortex sheets for 2D incompressible Euler equation, which are supported on several small closed curves. These examples represent a new type of vortex sheet solutions other than two known classes. The construction is based on Birkhoff-Rott operator, and accomplished by using implicit function theorem at point vortex solutions with suitably chosen function spaces.
	\end{abstract}
	
	\maketitle{\small{\bf Keywords:} The Euler equation, vortex sheets, Birkhoff-Rott operator, implicit function theorem \\
		
	}

\section{Introduction and main results}

The vorticity formulation of 2D incompressible Euler equation in the whole space reads as follows
\begin{align}\label{1-1}
	\begin{cases}
		\partial_t\vartheta+\mathbf{v}\cdot \nabla \vartheta =0&\text{in}\ \mathbb{R}^2\times (0,T),
		\\
		\vartheta\big|_{t=0}=\vartheta_0 &\text{in}\ \mathbb{R}^2,
	\end{cases}
\end{align}
where $\mathbf{v}=(v_1,v_2)$ is the velocity field, and $\vartheta=curl \mathbf{v}:=\partial_1v_2-\partial_2v_1$ is the vorticity function describing the rotating of fluid. Using the incompressible condition $\nabla\cdot \mathbf{v}=0$, the velocity can be recovered by the Biot-Savart law from $\vartheta$:
$$
\ \mathbf{v}=\nabla^\perp(-\Delta)^{-1}\vartheta,
$$
where $(x_1,x_2)^\perp=(x_2,-x_1)$ and the operator $(-\Delta)^{-1}$ is given by
\begin{equation*}
	(-\Delta)^{-1}\vartheta(x)=\int_{\mathbb{R}^2}G(x- y)\vartheta(y)dy.
\end{equation*}
Here $G(x,y)=\frac{1}{2\pi}\ln \frac{1}{|x-y|}$ is the fundamental solution of $-\Delta$ in $\mathbb{R}^2$.

In this paper, we are concerned with the existence of co-rotating and traveling vortex sheet solutions to \eqref{1-1}. Generally speaking, a vortex sheet is a weak solution in $\mathcal M(\mathbb R^2)\cap H^{-1}(\mathbb R^2)$ concentrated on simple closed curves $\Gamma_i:=\{\boldsymbol z_i(t,x)\in \mathbb R^2 \, : \, x\in [0,2\pi) \}$ with vortex strength $\omega_i$, namely
\begin{equation*}
	\vartheta(t,\boldsymbol x)=\sum_{i=0}^{m-1}\omega_i(t,x)\boldsymbol\delta_{\Gamma_i(t,\boldsymbol z)},
\end{equation*}
where $\omega_i(t,x)$ is the vorticity strength with respect to the parametrization $\boldsymbol z_i$, and $\boldsymbol\delta_{\Gamma}$ is the Dirac measure on the curve $\Gamma$. The above form of vorticity is understood in the sense that
\begin{equation*}
	\int_{\mathbb R^2}\varphi(x)d\vartheta(t, \boldsymbol x)=\sum_{i=0}^{m-1}\int_0^{2\pi} \varphi(\boldsymbol z_i(t,x))\omega_i(t,x)dx, \ \ \ \ \ \ \forall \, \varphi(x)\in C_c^\infty(\mathbb R^2).
\end{equation*}

In the formulation of a vortex sheet solution $\vartheta(t,\boldsymbol x)$, we take $\gamma(t,x)=\omega(t,x)/|\partial_x\boldsymbol z(t,x)|$ be the regularized vortex strength. According to the Birkhoff-Rott operator, the velocity field generated by the vortex sheet $\Gamma$ is given by
\begin{equation}\label{1-2}
	\mathbf v(\boldsymbol z(t,x),t)=-\mathcal{BR}(\boldsymbol z,\gamma):=-\frac{1}{2\pi}PV\int_0^{2\pi}\frac{(\boldsymbol z(t,x)-\boldsymbol z(t,y))^\perp}{|\boldsymbol z(t,x)-\boldsymbol z(t,y)|^2}\gamma(t,y)dy,
\end{equation}
with $PV$ the notation of principle value. In the fluid near a sheet, the respective limits of velocity on one side of $\Gamma_i$ is denoted by $\mathbf v_i^+$, while on the other is denoted by $\mathbf v_i^-$. Although the velocity field has a ``jump" across $\Gamma_i$ which causes some kind of singularity, we can define the velocity $\mathbf v_{\Gamma_i}$ on $\Gamma_i$ by imposing $\mathbf v_{\Gamma_i}=(\mathbf v_i^++\mathbf v_i^-)/2$. Using above notations, the evolution of $\boldsymbol z(t,x)$ and $\gamma(t,x)$ can be described by
\begin{equation}\label{1-3}
	\partial_t \boldsymbol z(t,x)=\mathbf v(\boldsymbol z,t)+c(x,t)\partial_x\boldsymbol z(x,t),
\end{equation}
\begin{equation}\label{1-4}
	\partial_t\gamma(t,x)=\partial_x(c(t,x)\gamma(t,x)),
\end{equation}
where $c(t,x)$ is the reparametrization freedom of the sheet curve.

The study of vortex sheets arises from the phenomenon of discontinuity for velocity field across an interface, which is caused by small viscosity flow passing rigid walls or sharp corners. It is observed that although the viscosity plays an important role in
the generation of vortex sheets, the later transportation and evolution of vortices can be modeled by the incompressible Euler equation \cite{Bir,MB}. However, the vortex sheet structure is proved to be extremely unstable, which leads to coherent vortex pattern often found in mixing layers, jets and wakes.

Before discussing our main topic of constructing special vortex sheets, let us first take some time to review the effort on existence of solutions to \eqref{1-1} in different settings: In 1963, Yudovich \cite{Yud} proved the global well-posedness of \eqref{1-1} with the initial data in $L^1\cap L^\infty$. This result covers the ``vortex patch" initial data solution, but is not applicable to more singular cases. For an $L_{\text{loc}}^2$ initial data whose vorticity has a definite sign, Delort \cite{Del} and Majda \cite{Maj} proved the global existence of a weak solution separately. When the vorticity does not have a definite sign, by imposing reflection symmetry on the initial vorticity, the global existence was proved by Lopes Filho et al. in \cite{Lop}. In the case that initial data $\boldsymbol z(0,x)$, $\gamma(0,x)$ for a vortex sheet are analytic, Sulem et al. \cite{Sul} obtained local existence for the analytic solution, which corresponds to the cases we discuss in this paper. On the other hand, Birkhoff et al. \cite{Bir1,Bir2} conjectured that curvature may blow up in finite time for vortex sheets with analytic initial data. This was verified by asymptotic analysis of Moore \cite{Moo} and numerical simulations of Krasny \cite{Kra}, Meiron et al. \cite{Mei}. It is notable that the system \eqref{1-3} and \eqref{1-4} was proved to be ill-posed in $H^s$ for $s>3/2$ in \cite{Caf}. For more results on the well-posedness theory, we refer to \cite{MB,Saf,Wu} and references therein.

 Concerning the existence of global solutions (also known as relative equilibria) to \eqref{1-1}, many vivid examples whose vorticity is supported on compact areas are known, which can be roughly divided into two classes: one contains solutions with relatively large support, which are bifurcated from known solutions including disks, annuli, Kirchhoff ellipses or Lamb dipole, see \cite{Cas3,de2,Hmi2}; the other contains solutions with highly concentrated vorticity, which can be regarded as a regularization of point vortex solutions, see \cite{Cao2,Cao3,HM}. Different approaches were developed for construction of this type, such as study of contour dynamic equation, variational method, Lyapunov-Schmidlt reduction, etc.

However, for equation \eqref{1-1}, there are only two known classes of relative equilibria composed of vortex sheets other than circles and lines before our results which will be given later in this paper(see Theorem \ref{thm1} and Theorem \ref{thm2}). The first one is Protas-Sakajo class \cite{Pro}, which is made out of segments rotating about a common center of rotation with endpoints at the vertices of a regular polygon. Duchon and Robert \cite{Duc} discussed the global existence when the initial data near one example constructed in \cite{Bir1}, where the vorticity is supported on a segment of length $2l$ with density
\begin{equation*}
	\gamma(x)=\Omega\sqrt{l^2-x^2}, \ \ \ \ \ \ \text{for} \ \ x\in[-l,l]
\end{equation*}
and angular velocity $\Omega$. ($\Omega>0$ for counterclockwise)

Recently, G\'{o}mez-Serrano et al. \cite{Gom1} gave another vortex sheet supported on a closed curve. By Lyapunov-Schmidlt reduction and degenerate bifurcation, they obtained two curves of solution bifurcated from one trivial solution, which is a unit circle at specific angular velocity. This remarkable result indicates that solutions bifurcated from concentric circles can be constructed in the same way, and constitute the second class of vortex sheet solutions. In their follow-up work \cite{Gom2}, they have further showed that one can not expect vortex sheets other than concentric circles when $\omega\ge 0$ and $\Omega\le0$.

In the present paper, using the Birkhoff-Rott operator \eqref{1-1}, we will give a new class of vortex sheet solutions to \eqref{1-1}, whose support is on several small closed curves. Our strategy is to transform \eqref{1-3} and \eqref{1-4} into an integro-differential system, where well-located point vortices are observed to be trivial solutions. By extending the range of parameter $\varepsilon$ to a neighborhood of $0$, we hope that there exist nontrivial vortex sheet solutions near the trivial ones. However, having done the local linearization, we find the Gateaux derivatives do not give an isomorphism, which makes the implicit function theorem invalid. To obtain local curves of nontrivial vortex sheet solutions, a few techniques of subtle analysis are adopted. We will introduce a series of quotient function spaces for this problem, and let the angular velocity $\Omega$ and traveling speed $W$ vary with the unknown pair $(\boldsymbol z,\gamma)$. Then the isomorphism condition for linearized operator will prove to be equivalent to a dynamic relationship, which exactly corroborates our physical intuition. From this point of view, the existence of co-rotating and traveling vortex sheets can be established by the Implicit function theorem in a refined version. Since Dirac measure at one point is in $H^{-1-\tau}(\mathbb R^2)$ for any $\tau>0$, which is more singular than a vortex sheet, our result can also be regarded as regularization of point vortices in a special way.

To transform the construction of vortex sheets into a nonlinear problem of our preference, we introduce some necessary notations. For $\varepsilon\in (0,1/2)$, we denote
$$\Gamma^\varepsilon:=\{\boldsymbol z(t,x)\in \mathbb R^2 \, : \, x\in [0,2\pi) \}$$
 as a simple closed curve encircling the origin, which is close to the circle with radius $\varepsilon$ centered at the origin.  In the rest of this paper, the unit vector in $x_i$ direction ($i=1,2$) is denoted by $\boldsymbol e_i$. We also use $Q_\theta$ to denote the counterclockwise rotation operator of angle $\theta$ with respect to the origin, and use
$$\int\!\!\!\!\!\!\!\!\!\; {}-{} g(\tau)d\tau:=\frac{1}{2\pi}\int_0^{2\pi}g(\tau)d\tau$$
to denote the mean value of integral on the unit circle. For simplicity, we will omit the principal value notation $PV$ for possible singularities in the integral.

The first kind of relative equilibria we are looking for are co-rotating sheet solutions. Suppose the solution $\vartheta_\varepsilon$ is composed of $m$-fold symmetric sheets rotating in a uniform angular velocity $\Omega$ centered at $(d,0)$, that is
\begin{equation}\label{1-5}
	\vartheta_\varepsilon(t,\boldsymbol x)=\vartheta_{0,\varepsilon}\left(Q_{-\Omega t}(\boldsymbol x-d\boldsymbol e_1)+d\boldsymbol e_1\right)=\frac{1}{\varepsilon}\sum\limits_{i=0}^{m-1}\omega(t,x)\boldsymbol\delta_{\Gamma^\varepsilon_i(t,\boldsymbol z)},
\end{equation}
where $\Gamma_i^\varepsilon(t,\boldsymbol z)\subset \mathbb{R}^2$ are closed curves satisfying
\begin{equation*}
	\Gamma_i^\varepsilon(t,\boldsymbol z)-d\boldsymbol e_1=Q_{\frac{2\pi i}{m}-\Omega t}\left(\Gamma^\varepsilon(\boldsymbol z)-d\boldsymbol e_1\right),\,\,i=0,\cdots,m-1
\end{equation*}
with some $d>1$ fixed. In this rotating frame, the evolution equations \eqref{1-3} and \eqref{1-4} become
\begin{equation}\label{1-6}
	\partial_t \boldsymbol z(t,x)=\mathbf v(t,\boldsymbol z)+\Omega(\boldsymbol z(t,x)-d\boldsymbol e_1)+c(t,x)\partial_x\boldsymbol z(t,x),	
\end{equation}
and
\begin{equation}\label{1-7}
	\partial_t\gamma(t,x)=\partial_x(c(t,x)\gamma(t,x)),
\end{equation}
with $c(t,x)$ being the reparameterization freedom introduced before. Since $\boldsymbol z(t,x)$ parameterizes the same curve $\Gamma^\varepsilon$ as $\boldsymbol z(0,x)$, $\partial_t \boldsymbol z(t,x)$ is tangent to $\Gamma^\varepsilon$. Let $\mathbf n(\boldsymbol z(t,x))$ be the unit outerward vector at $\boldsymbol z(t,x)$, and $\mathbf s(\boldsymbol z(t,x))$ be the unit tangent vector counterclockwise.  Multiplying \eqref{1-6} by $\mathbf n(\boldsymbol z(t,x))$  and using $\mathbf n(\boldsymbol z(t,x))\cdot \partial_x\boldsymbol z(t,x)=0$, we obtain
\begin{equation}\label{1-8}
	\left(\mathbf{v}(t,\boldsymbol z)+\Omega(\boldsymbol z(t,x)-d\boldsymbol e_1)^\perp\right)\cdot \mathbf n(\boldsymbol z(t,x))=0, \ \ \ \forall \, \boldsymbol z(t,x) \in\Gamma^\varepsilon.
\end{equation}
To eliminate time parameter and derive another necessary relationship, we choose
\begin{equation*}
	c(t,x)=-\frac{(\mathbf{v}(t,\boldsymbol z)+\Omega(\boldsymbol z(t,x)-d\boldsymbol e_1)^\perp)\cdot \mathbf s(\boldsymbol z(t,x))}{| \partial_x\boldsymbol z(t,x)|},
\end{equation*}
and hence obtain $\partial_t\boldsymbol z(t,x)\cdot \mathbf s(\boldsymbol z(t,x))=0$ by multiplying $\mathbf s(\boldsymbol z(t,x))$ to \eqref{1-6}. Note that $\mathbf s(\boldsymbol z(t,x))$ is parallel to $\partial_t\boldsymbol z(t,x)$ and has unit length. So it holds $\partial_t\boldsymbol z(t,x)=0$, which implies that $\boldsymbol z(t,x)=\boldsymbol z(x)$ remains fixed in time. On the other hand, $\vartheta_\varepsilon$ is stationary in the rotating frame. Since $\boldsymbol z(t,x)=\boldsymbol z(x)$ is invariant, the regularized vortex strength will not evolve with time, or namely $\gamma(t,x)=\gamma(x)$. From \eqref{1-7}, we derive
\begin{equation*}
	c(t,x)\gamma(t,x)=C.
\end{equation*}
According to the definition of $c(t,x)$ given above, we have
\begin{equation}\label{1-9}
	\frac{(\mathbf{v}(\boldsymbol z)+\Omega(\boldsymbol z(x)-d\boldsymbol e_1)^\perp)\cdot \mathbf s(\boldsymbol z(x))\gamma(x)}{| \partial_x\boldsymbol z(x)|}=-C, \ \ \ \forall \, x\in [0,2\pi).
\end{equation}
Now we can assume that $\boldsymbol z(x)$ is parameterized as
$$\boldsymbol{z}(x)=\left(\varepsilon R(x)\cos(x), \varepsilon R(x)\sin(x)\right)$$ for some function $R(x)$ close to $1$. From \eqref{1-8} \eqref{1-9} and the definition of Birkhoff-Rott operator \eqref{1-2}, if we let
\begin{equation*}
	\begin{split}
	    F_1&(\varepsilon,\Omega,R,\gamma)=\Omega(\varepsilon R(x)R'(x)-dR'(x)\cos(x)+dR(x)\sin(x)) \\
	    &+\frac{1}{\varepsilon}\int\!\!\!\!\!\!\!\!\!\; {}-{} \frac{R(x)R'(x)-R'(x)R(y)\cos(x-y)+R(x)R(y)\sin(x-y)}{(R(x)-R(y))^2+4R(x)R(y)\sin^2(\frac{x-y}{2})}\times\gamma(y)dy\\
	    &+\sum_{i=1}^{m-1}\int\!\!\!\!\!\!\!\!\!\; {}-{} \frac{R(x)R'(x)-R'(x)R(y)\cos(x-y-\frac{2\pi i}{m})+R(x)R(y)\sin(x-y-\frac{2\pi i}{m})}{\left| \left(\boldsymbol{z}(x)-(d,0)\right)-Q_{\frac{2\pi i}{m}}\left(\boldsymbol{z}(y)-(d,0)\right)\right|^2}\times\gamma(y)dy\\
	    &-\sum_{i=1}^{m-1}\int\!\!\!\!\!\!\!\!\!\; {}-{} \frac{dR(x)\left(1-\cos(\frac{2\pi i}{m})\right)\sin(x)-dR'(x)\left(1-\cos(\frac{2\pi i}{m})\right)\cos(x)}{\left| \left(\boldsymbol{z}(x)-(d,0)\right)-Q_{\frac{2\pi i}{m}}\left(\boldsymbol{z}(y)-(d,0)\right)\right|^2}\times\gamma(y)dy,
	\end{split}
\end{equation*}
and
\begin{equation*}
	\begin{split}
		\tilde F_2&(\varepsilon,\Omega,R,\gamma)=\frac{\gamma(x)\Omega}{R'(x)^2+R(x)^2}\cdot(-\varepsilon R(x)^2+dR(x)\cos(x)+dR'(x)\sin(x))\\
		&+\frac{1}{\varepsilon}\frac{\gamma(x)}{R'(x)^2+R(x)^2}\int\!\!\!\!\!\!\!\!\!\; {}-{} \frac{-R(x)^2+R(x)R(y)\cos(x-y)+R'(x)R(y)\sin(x-y)}{(R(x)-R(y))^2+4R(x)R(y)\sin^2(\frac{x-y}{2})}\times\gamma(y)dy\\
		&+\sum_{i=1}^{m-1}\frac{\gamma(x)}{R'(x)^2+R(x)^2}\int\!\!\!\!\!\!\!\!\!\; {}-{} \frac{-R(x)^2+R(x)R(y)\cos(x-y-\frac{2\pi i}{m})+R'(x)R(y)\sin(x-y-\frac{2\pi i}{m})}{\left| \left(\boldsymbol{z}(x)-(d,0)\right)-Q_{\frac{2\pi i}{m}}\left(\boldsymbol{z}(y)-(d,0)\right)\right|^2}\times\gamma(y)dy\\
		&-\sum_{i=1}^{m-1}\frac{\gamma(x)}{R'(x)^2+R(x)^2}\int\!\!\!\!\!\!\!\!\!\; {}-{} \frac{dR(x)\left(1-\cos(\frac{2\pi i}{m})\right)\cos(x)+dR'(x)\left(1-\cos(\frac{2\pi i}{m})\right)\sin(x)}{\left| \left(\boldsymbol{z}(x)-(d,0)\right)-Q_{\frac{2\pi i}{m}}\left(\boldsymbol{z}(y)-(d,0)\right)\right|^2}\times\gamma(y)dy,
	\end{split}
\end{equation*}
then the construction of co-rotating solutions is transformed to finding $R(x)$ and $\gamma(x)$ such that
\begin{equation}\label{1-10}
\left\{
\begin{array}{l}
	F_1(\varepsilon,\Omega,R,\gamma)=0, \\
F_2(\varepsilon,\Omega,R,\gamma)=(I-\mathcal P)\tilde F_2=0,\\
\end{array}
\right.
\end{equation}
where $\mathcal P$ is the operator of projection to the mean, namely, $\mathcal Pg:=\int\!\!\!\!\!\!\!\!\!\; {}-{} g(x)dx$.

By choosing suitable quotient function spaces, we will show that the structure of $F_1$ and $F_2$ allows us to use implicit function theorem at $(\varepsilon,\Omega, R,\gamma)=(0,\Omega^*, 1,1)$ with
\begin{equation}\label{1-11}
	\Omega^*:=\frac{m-1}{2d^2},
\end{equation}
and thus establish the existence of a family of co-rotating vortex sheet solutions to \eqref{1-1} generated from $(0,\Omega^*, 1,1)$. One can also verify that $(0,\Omega^*, 1,1)$ corresponds to co-rotating $m$-fold point vortices solutions to \eqref{1-1}, where the intensity of every single point vortex is $2\pi$, the perimeter of unit circle. Now we can state our first result.
\begin{theorem}\label{thm1}
	Suppose that
$m\ge 2$. Then there exists $\varepsilon_0>0$ such that for any $\varepsilon\in (0,\varepsilon_0)$, \eqref{1-1} has a global co-rotating vortex sheet solution $\vartheta_\varepsilon(\boldsymbol x,t)=\vartheta_{0,\varepsilon}\left(Q_{-\Omega t}(\boldsymbol x-d\boldsymbol e_1)+d\boldsymbol e_1\right)$ defined in \eqref{1-5}, and the uniform angular velocity $\Omega_\varepsilon$ satisfies
	$$\Omega=\Omega^*+O(\varepsilon)$$
	with $\Omega^*$ given in \eqref{1-11}. Moreover, each component of the solution is supported on a $C^\infty$ closed curve with convex interior.
\end{theorem}

 Next we turn to the existence of traveling sheet solutions. In this case, the solution $\vartheta_\varepsilon$ is composed of two mirror symmetric sheets with opposite sign traveling at a uniform speed $W>0$ in $x_2$-direction, namely
\begin{equation}\label{1-12}
	\vartheta_\varepsilon(t,\boldsymbol x)=\vartheta_{0,\varepsilon}(\boldsymbol x-tW\boldsymbol e_2)=\frac{1}{\varepsilon}\omega(t,x)\boldsymbol\delta_{\Gamma^\varepsilon_O(t,\boldsymbol z)}-\frac{1}{\varepsilon}\omega(t,x)\boldsymbol\delta_{\Gamma^\varepsilon_T(t,\boldsymbol z)},
\end{equation}
where $\Gamma_O^\varepsilon(t,\boldsymbol z), \Gamma_T^\varepsilon(t,\boldsymbol z)\subset \mathbb{R}^2$ satisfy
\begin{equation*}
	\Gamma_O^\varepsilon(t,\boldsymbol z)=\Gamma^\varepsilon(\boldsymbol z)+tW\boldsymbol e_2,
\end{equation*}
and
\begin{equation*}
	\Gamma_T^\varepsilon(t,\boldsymbol z)=-\Gamma^\varepsilon(\boldsymbol z)+2d\boldsymbol e_1+tW\boldsymbol e_2
\end{equation*}
with $d>1$ the half distance of two sheets center. By a similar method as we investigate co-rotating sheet solutions, we can make $\vartheta_\varepsilon$ stationary in the translation frame, and $\boldsymbol z(t,x),\gamma(t,x)$ invariant of time. Hence the two dynamic equations for traveling sheet solutions are
\begin{equation}\label{1-13}
   \left(\mathbf{v}(\boldsymbol z)-W\boldsymbol e_2\right)\cdot \mathbf n(\boldsymbol z(x))=0, \ \ \ \forall \,  x\in [0,2\pi),
\end{equation}
and
\begin{equation}\label{1-14}
	\frac{(\mathbf{v}(\boldsymbol z)+\Omega(\boldsymbol z(x)-d\boldsymbol e_1)^\perp)\cdot \mathbf s(\boldsymbol z(x))\gamma(x)}{| \partial_x\boldsymbol z(x)|}=-C, \ \ \ \forall \, x\in [0,2\pi).
\end{equation}
By doing a same parameterization for $\boldsymbol z(x)$ as before and using Birkhoff-Rott operator \eqref{1-2}, we see that \eqref{1-13} and \eqref{1-14} are equivalent to
\begin{equation}\label{1-15}
\left\{
\begin{array}{l}
	G_1(\varepsilon,W,R,\gamma)=0, \\
 G_2(\varepsilon,W,R,\gamma)=(I-\mathcal P)\tilde G_2=0,
 \end{array}
 \right.
\end{equation}
where
\begin{equation*}
	\begin{split}
		G_1&(\varepsilon,W,R,\gamma)=-W\big(R(x)\sin(x)-R'(x)\cos(x)\big) \\
		&+\frac{1}{\varepsilon}\int\!\!\!\!\!\!\!\!\!\; {}-{} \frac{R(x)R'(x)-R'(x)R(y)\cos(x-y)+R(x)R(y)\sin(x-y)}{(R(x)-R(y))^2+4R(x)R(y)\sin^2(\frac{x-y}{2})}\times\gamma(y)dy\\
		&-\int\!\!\!\!\!\!\!\!\!\; {}-{} \frac{R(x)R'(x)+R'(x)R(y)\cos(x-y)-R(x)R(y)\sin(x-y)}{\left|(\varepsilon R(x)\cos(x)+\varepsilon R(y)\cos(y)-2d)^2+(\varepsilon R(x)\sin(x)+\varepsilon R(y)\sin(y))^2\right|^2}\times\gamma(y)dy\\
		&+\int\!\!\!\!\!\!\!\!\!\; {}-{} \frac{2dR(x)\sin(x)-2dR'(x)\cos(x)}{\left|(\varepsilon R(x)\cos(x)+\varepsilon R(y)\cos(y)-2d)^2+(\varepsilon R(x)\sin(x)+\varepsilon R(y)\sin(y))^2\right|^2}\times\gamma(y)dy,
	\end{split}
\end{equation*}
and
\begin{equation*}
	\begin{split}
		\tilde G_2&(\varepsilon,W,R,\gamma)=-\frac{\gamma(x)W}{R'(x)^2+R(x)^2}\cdot\big(R'(x)\sin(x)+R(x)\cos(x)\big)\\
		&+\frac{1}{\varepsilon}\frac{\gamma(x)}{R'(x)^2+R(x)^2}\int\!\!\!\!\!\!\!\!\!\; {}-{} \frac{-R(x)^2+R(x)R(y)\cos(x-y)+R'(x)R(y)\sin(x-y)}{(R(x)-R(y))^2+4R(x)R(y)\sin^2(\frac{x-y}{2})}\times\gamma(y)dy\\
		&+\frac{\gamma(x)}{R'(x)^2+R(x)^2}\int\!\!\!\!\!\!\!\!\!\; {}-{} \frac{R(x)^2+R(x)R(y)\cos(x-y)+R'(x)R(y)\sin(x-y)}{\left|(\varepsilon R(x)\cos(x)+\varepsilon R(y)\cos(y)-2d)^2+(\varepsilon R(x)\sin(x)+\varepsilon R(y)\sin(y))^2\right|^2}\gamma(y)dy\\
		&+\frac{\gamma(x)}{R'(x)^2+R(x)^2}\int\!\!\!\!\!\!\!\!\!\; {}-{} \frac{dR(x)\left(1-\cos(\frac{2\pi i}{m})\right)\cos(x)+dR'(x)\left(1-\cos(\frac{2\pi i}{m})\right)\sin(x)}{\left|(\varepsilon R(x)\cos(x)+\varepsilon R(y)\cos(y)-2d)^2+(\varepsilon R(x)\sin(x)+\varepsilon R(y)\sin(y))^2\right|^2}\gamma(y)dy.
	\end{split}
\end{equation*}

We can use implicit function theorem at $(\varepsilon,W, R,\gamma)=(0,W^*, 1,1)$ with
\begin{equation}\label{1-16}
	W^*:=\frac{1}{2d}
\end{equation}
  as before. Careful analysis gives a curve of nontrivial traveling vortex sheets passing the trivial one $(0,W^*, 1,1)$, which represents the traveling point vortex pair with opposite intensity $\pm2\pi$. This fact leads to our second theorem:
\begin{theorem}\label{thm2}
	There exists $\varepsilon_0>0$ such that for any $\varepsilon\in (0,\varepsilon_0)$, \eqref{1-1} has a global traveling vortex sheet solution $\vartheta_\varepsilon(\boldsymbol x,t)=\vartheta_{0,\varepsilon}(\boldsymbol x-tW\boldsymbol e_2)$ defined in \eqref{1-12}, and the uniform traveling speed $W$ satisfies
	$$W=W^*+O(\varepsilon).$$
    with $W^*$ given in \eqref{1-16}. Moreover, each component of the solution is supported on a $C^\infty$ closed curve with convex interior.
\end{theorem}

\begin{remark}
	Since $R(x)$ and $\gamma(x)$ will be chosen in analytic spaces, the support curves for vortex sheets obtained in Theorem \ref{thm1} and \ref{thm2} are not only $C^\infty$ smooth, but also real analytic. For vortex patch in 2D incompressible Euler flow, the regularity of boundary is an interesting topic. Chemin \cite{Che} and Bertozzi et al. \cite{Ber} showed that the patch boundary will persist $C^{1,\alpha}$ regularity. As a special case, Hmidi et al. \cite{Hmi} proved the $C^\infty$ boundary regularity for relative equilibria bifurcated from disks. Different from the patch case, the regularity will not be preserved for a general vortex sheet intial data, see \cite{Bir,Bir1}.
\end{remark}

 It is notable that since there is no semilinear elliptic problem corresponding to vortex sheet, traditional variational method for vortex patch does not work here. Our method for concentrative vortex sheets is novel and widely applicable: various solutions can be constructed in this way, where each local cluster of sheet is supported on small concentric circles or has other different locations. However, the system will be more complicated in foresaid situations and it takes much computation time. We also bring our readers' attention to the fact that besides the existence, the local uniqueness for relative equilibria of such type remains another open problem.

 This paper is organized, as follows: In Section 2, we introduce the functional settings and a special form of Taylor's formula. Then we will prove the existence of co-rotating vortex sheets in Section 3. The complete proof contains three essential elements: the verification of regularity for the integro-differential system, the linearization of system, and the application of implicit function theorem. In Section 4, the existence of traveling vortex sheets will be given by a brief discussion, since it is similar to that of the co-rotating case.

\section{Functional settings and Taylor's formula}
For some small constant $a>0$, let $C_w(a)$ denote the space of analytic functions in the strip $|\mathbf{Im} (z)|\le a$. To apply the implicit function theorem, we define the even function spaces
\begin{equation*}
	X^k=\left\{ u\in C_w(a), \ u(x)= \sum\limits_{j=1}^{\infty}a_j\cos(jx), \ \sum_{\pm}\int_0^{2\pi}|u(x\pm a\mathbf{i})|^2+|\partial^k u(x\pm a\mathbf{i})|^2dx\le +\infty\right\},
\end{equation*}
and the odd function spaces
\begin{equation*}
	Y^k=\left\{ u\in C_w(a), \ u(x)= \sum\limits_{j=1}^{\infty}a_j\sin(jx), \ \sum_{\pm}\int_0^{2\pi}|u(x\pm a\mathbf{i})|^2+|\partial^k u(x\pm a\mathbf{i})|^2dx\le +\infty\right\},
\end{equation*}
where the norm of $X^k$, $Y^k$ is the standard $H^k$ norm on the strip $|\mathbf{Im} (z)|\le a$, namely, for some $u_1\in X^k$ and $u_2\in Y^k$
\begin{equation*}
	\|u_1\|_{X^k}=\|u_1(x\pm a\mathbf{i})\|_{H^k} \ \ \ \text{and} \ \ \ 	\|u_2\|_{Y^k}=\|u_2(x\pm a\mathbf{i})\|_{H^k}.
\end{equation*}
Our construction has a distinguished feature: due to the spectral structure on Fourier series, the linearized operator is no longer an isomorphism on above spaces. To remediate to this weak point, we define the quotient function spaces
$$\mathcal {A}^{k}:=\Bigg\{(u, v)\in X^{k+1}\times X^k \Big| \int\!\!\!\!\!\!\!\!\!\; {}-{} u(x)\cos(x)dx=-\int\!\!\!\!\!\!\!\!\!\; {}-{} v(x)\cos(x)dx\Bigg\},$$
$$\mathcal {B}^{k}:=\Bigg\{(u, v)\in Y^k\times X^k \Big| \int\!\!\!\!\!\!\!\!\!\; {}-{} u(x)\sin(x)dx=-\int\!\!\!\!\!\!\!\!\!\; {}-{} v(x)\cos(x)dx\Bigg\},$$
which give a restriction on the first Fourier coefficient. On the other hand, in the first step we extend the range of $\varepsilon$ from $(0,1/2)$ to $(-1/2,1/2)$, where the functionals should prove to be well-defined. To eliminate possible singularity and derive asymptotic expansion for $F_1,F_2$ and $G_1,G_2$, we will frequently use the following Taylor's formula in the proof of regularity part.
\begin{equation}\label{2-1}
	\frac{1}{(A+B)^\lambda}=\frac{1}{A^\lambda}-\lambda\int_0^1\frac{B}{(A+tB)^{1+\lambda}}dt.
\end{equation}

In the next two sections, we will consider the existence of co-rotating and traveling vortex sheets solutions separately. Since the problem is transformed to solving an integro-differential system, we will focus on dealing with relevant functionals.

\section{Existence of co-rotating  vortex sheets}
To normalize our problem, we extend the range of $\varepsilon$ to $(-\frac{1}{2},\frac{1}{2})$. Suppose that $R(x)=1+\varepsilon f(x)$, $\gamma(x)=1+\varepsilon g(x)$
with $x\in [0,2\pi)$, and $f, g$ being two $C^1$ functions. From discussion in the previous section, we will consider following functionals
\begin{equation*}
	F_1(\varepsilon,\Omega,f,g)=F_{11}+F_{12}+F_{13}+F_{14},
\end{equation*}
where
\begin{equation}\label{2-2}
	F_{11}=\Omega\left(\varepsilon^2(1+\varepsilon f(x))f'(x)-\varepsilon f'(x)d\cos(x)+(1+\varepsilon f(x))d\sin(x)\right),
\end{equation}

\begin{equation}\label{2-3}
	\begin{split}
	F_{12}&=\frac{1}{\varepsilon}\int\!\!\!\!\!\!\!\!\!\; {}-{} \frac{\varepsilon(1+\varepsilon f(x))f'(x)(1-\cos(x-y))(1+\varepsilon g(y))}{ \varepsilon^2\left(f(x)-f(y)\right)^2+4(1+\varepsilon f(x))(1+\varepsilon f(y))\sin^2\left(\frac{x-y}{2}\right)} dy\\
	& \ \ \ +\frac{1}{\varepsilon}\int\!\!\!\!\!\!\!\!\!\; {}-{} \frac{\varepsilon^2 f'(x)(f(x)-f(y))\cos(x-y)(1+\varepsilon g(y))}{ \varepsilon^2\left(f(x)-f(y)\right)^2+4(1+\varepsilon f(x))(1+\varepsilon f(y))\sin^2\left(\frac{x-y}{2}\right)}dy\\
	& \ \ \ -\frac{1}{\varepsilon}\int\!\!\!\!\!\!\!\!\!\; {}-{} \frac{\varepsilon(1+\varepsilon f(x))(f(x)-f(y))\sin(x-y)(1+\varepsilon g(y))}{ \varepsilon^2\left(f(x)-f(y)\right)^2+4(1+\varepsilon f(x))(1+\varepsilon f(y))\sin^2\left(\frac{x-y}{2}\right)}dy\\
	& \ \ \ +\frac{1}{\varepsilon}\int\!\!\!\!\!\!\!\!\!\; {}-{} \frac{(1+\varepsilon f(x))^2\sin(x-y)(1+\varepsilon g(y))}{ \varepsilon^2\left(f(x)-f(y)\right)^2+4(1+\varepsilon f(x))(1+\varepsilon f(y))\sin^2\left(\frac{x-y}{2}\right)}dy\\
	&=F_{121}+F_{122}+F_{123}+F_{124},
	\end{split}
\end{equation}

\begin{equation}\label{2-4}
	\begin{split}
		F_{13}&=\sum_{i=1}^{m-1}\int\!\!\!\!\!\!\!\!\!\; {}-{} \frac{\varepsilon(1+\varepsilon f(x))f'(y)(1+\varepsilon g(y))}{\left| \left(\boldsymbol{z}(x)-(d,0)\right)-Q_{\frac{2\pi i}{m}}\left(\boldsymbol{z}(y)-(d,0)\right)\right|^2}dy\\
		&\ \ \ -\sum_{i=1}^{m-1}\int\!\!\!\!\!\!\!\!\!\; {}-{} \frac{\varepsilon f'(x)(1+\varepsilon f(y))\cos(x-y-\frac{2\pi i}{m})(1+\varepsilon g(y))}{\left| \left(\boldsymbol{z}(x)-(d,0)\right)-Q_{\frac{2\pi i}{m}}\left(\boldsymbol{z}(y)-(d,0)\right)\right|^2}dy\\
		& \ \ \ +\sum_{i=1}^{m-1}\int\!\!\!\!\!\!\!\!\!\; {}-{} \frac{(1+\varepsilon f(x))(1+\varepsilon f(x))\sin(x-y-\frac{2\pi i}{m})(1+\varepsilon g(y))}{\left| \left(\boldsymbol{z}(x)-(d,0)\right)-Q_{\frac{2\pi i}{m}}\left(\boldsymbol{z}(y)-(d,0)\right)\right|^2}dy,
	\end{split}
\end{equation}

\begin{equation}\label{2-5}
	\begin{split}
	 F_{14}&=\sum_{i=1}^{m-1}\int\!\!\!\!\!\!\!\!\!\; {}-{} \frac{d(1+\varepsilon f(x))\left(1-\cos(\frac{2\pi i}{m})\right)\sin(x)(1+\varepsilon g(y))}{\left| \left(\boldsymbol{z}(x)-(d,0)\right)-Q_{\frac{2\pi i}{m}}\left(\boldsymbol{z}(y)-(d,0)\right)\right|^2}dy\\
	& \ \ \ -\sum_{i=1}^{m-1}\int\!\!\!\!\!\!\!\!\!\; {}-{} \frac{\varepsilon df'(x)\left(1-\cos(\frac{2\pi i}{m})\right)\cos(x)(1+\varepsilon g(y))}{\left| \left(\boldsymbol{z}(x)-(d,0)\right)-Q_{\frac{2\pi i}{m}}\left(\boldsymbol{z}(y)-(d,0)\right)\right|^2}dy,
	\end{split}
\end{equation}
and
\begin{equation*}
	F_2(\varepsilon,\Omega,f,g)=(I-\mathcal P)\tilde F_2=(I-\mathcal P)(\tilde F_{21}+\tilde F_{22}+\tilde F_{23}+\tilde F_{24}),
\end{equation*}
where
\begin{equation}\label{2-6}
	\tilde F_{21}=\frac{\Omega(1+\varepsilon g(x))}{\varepsilon^2f'(x)^2+(1+\varepsilon f(x))^2}\cdot(-\varepsilon (1+\varepsilon f(x))^2+(1+\varepsilon f(x))d\cos(x)+\varepsilon f'(x)d\sin(x)),
\end{equation}

\begin{equation}\label{2-7}
	\begin{split}
		\tilde F_{22}&=\frac{1}{\varepsilon}\cdot\frac{(1+\varepsilon g(x))}{\varepsilon^2f'(x)^2+(1+\varepsilon f(x))^2}\int\!\!\!\!\!\!\!\!\!\; {}-{} \frac{(1+\varepsilon f(x))(1+\varepsilon f(y))(\cos(x-y)-1)(1+\varepsilon g(y))}{ \varepsilon^2\left(f(x)-f(y)\right)^2+4(1+\varepsilon f(x))(1+\varepsilon f(y))\sin^2\left(\frac{x-y}{2}\right)}dy\\
		& \ \ \ -\frac{1}{\varepsilon}\cdot\frac{(1+\varepsilon g(x))}{\varepsilon^2f'(x)^2+(1+\varepsilon f(x))^2}\int\!\!\!\!\!\!\!\!\!\; {}-{} \frac{\varepsilon^2 f'(x)(f(x)-f(y))\sin(x-y)(1+\varepsilon g(y))}{\varepsilon^2\left(f(x)-f(y)\right)^2+4(1+\varepsilon f(x))(1+\varepsilon f(y))\sin^2\left(\frac{x-y}{2}\right)}dy\\
		& \ \ \ -\frac{1}{\varepsilon}\cdot\frac{(1+\varepsilon g(x))}{\varepsilon^2f'(x)^2+(1+\varepsilon f(x))^2}\int\!\!\!\!\!\!\!\!\!\; {}-{} \frac{\varepsilon(1+\varepsilon f(x))(f(x)-f(y))(1+\varepsilon g(y))}{\varepsilon^2\left(f(x)-f(y)\right)^2+4(1+\varepsilon f(x))(1+\varepsilon f(y))\sin^2\left(\frac{x-y}{2}\right)}dy\\
		& \ \ \ +\frac{1}{\varepsilon}\cdot\frac{(1+\varepsilon g(x))}{\varepsilon^2f'(x)^2+(1+\varepsilon f(x))^2}\int\!\!\!\!\!\!\!\!\!\; {}-{} \frac{\varepsilon f'(x)(1+\varepsilon f(x))\sin(x-y)(1+\varepsilon g(y))}{\varepsilon^2\left(f(x)-f(y)\right)^2+4(1+\varepsilon f(x))(1+\varepsilon f(y))\sin^2\left(\frac{x-y}{2}\right)}dy\\
		&=F_{221}+F_{222}+F_{223}+F_{224},\\
	\end{split}
\end{equation}

\begin{equation}\label{2-8}
	\begin{split}
		\tilde F_{23}&=\frac{(1+\varepsilon g(x))}{\varepsilon^2f'(x)^2+(1+\varepsilon f(x))^2}\sum_{i=1}^{m-1}\int\!\!\!\!\!\!\!\!\!\; {}-{} \frac{(1+\varepsilon f(x))(1+\varepsilon f(y))\cos(x-y-\frac{2\pi i}{m})(1+\varepsilon g(y))}{\left| \left(\boldsymbol{z}(x)-(d,0)\right)-Q_{\frac{2\pi i}{m}}\left(\boldsymbol{z}(y)-(d,0)\right)\right|^2}dy\\
		& \ \ \ +\frac{(1+\varepsilon g(x))}{\varepsilon^2f'(x)^2+(1+\varepsilon f(x))^2}\sum_{i=1}^{m-1}\int\!\!\!\!\!\!\!\!\!\; {}-{} \frac{\varepsilon f'(x)(1+\varepsilon f(y))\sin(x-y-\frac{2\pi i}{m})(1+\varepsilon g(y))}{\left| \left(\boldsymbol{z}(x)-(d,0)\right)-Q_{\frac{2\pi i}{m}}\left(\boldsymbol{z}(y)-(d,0)\right)\right|^2}dy,\\
		& \ \ \ -\frac{(1+\varepsilon g(x))}{\varepsilon^2f'(x)^2+(1+\varepsilon f(x))^2}\sum_{i=1}^{m-1}\int\!\!\!\!\!\!\!\!\!\; {}-{} \frac{(1+\varepsilon f(x))^2(1+\varepsilon g(y))}{\left| \left(\boldsymbol{z}(x)-(d,0)\right)-Q_{\frac{2\pi i}{m}}\left(\boldsymbol{z}(y)-(d,0)\right)\right|^2}dy,
	\end{split}
\end{equation}

\begin{equation}\label{2-9}
	\begin{split}
		\tilde F_{24}&=-\frac{(1+\varepsilon g(x))}{\varepsilon^2f'(x)^2+(1+\varepsilon f(x))^2}\sum_{i=1}^{m-1}\int\!\!\!\!\!\!\!\!\!\; {}-{} \frac{d(1+\varepsilon f(x))\left(1-\cos(\frac{2\pi i}{m})\right)\cos(x)(1+\varepsilon g(y))}{\left| \left(\boldsymbol{z}(x)-(d,0)\right)-Q_{\frac{2\pi i}{m}}\left(\boldsymbol{z}(y)-(d,0)\right)\right|^2}dy\\
		& \ \ \ -\frac{(1+\varepsilon g(x))}{\varepsilon^2f'(x)^2+(1+\varepsilon f(x))^2}\sum_{i=1}^{m-1}\int\!\!\!\!\!\!\!\!\!\; {}-{} \frac{\varepsilon df'(x)\left(1-\cos(\frac{2\pi i}{m})\right)\sin(x)(1+\varepsilon g(y))}{\left| \left(\boldsymbol{z}(x)-(d,0)\right)-Q_{\frac{2\pi i}{m}}\left(\boldsymbol{z}(y)-(d,0)\right)\right|^2}dy.
	\end{split}
\end{equation}
We also denote $F_{2i}=(I-\mathcal{P})\tilde F_{2i}$ for $i=1,2,3,4$. The proof will be split into several lemmas, each of which is associated with checking the hypotheses of the implicit function theorem.

Let $V^r$ be the open neighborhood of zero in $\mathcal{A}^k$
\begin{equation*}
	V^r:=\left\{(f,g)\in \mathcal{A}^k : \  \|f\|_{X^{k+1}}<r, \  \|g\|_{X^k}<r\right\}
\end{equation*}
with $0<r<1$ and $k\ge 3$. In the following two lemmas, we investigate the regularity of $F_1(\varepsilon, \Omega, f,g)$, $F_2(\varepsilon, \Omega, f,g)$, and show the extension for $\varepsilon$ to $ \left(-\frac{1}{2}, \frac{1}{2}\right)$ is reasonable.
\begin{lemma}\label{lem2-1}
	$F_1(\varepsilon, \Omega, f, g): \left(-\frac{1}{2}, \frac{1}{2}\right)\times \mathbb{R} \times V^r \rightarrow Y^k$ and $F_2(\varepsilon, \Omega, f, g): \left(-\frac{1}{2}, \frac{1}{2}\right)\times \mathbb{R} \times V^r \rightarrow X^k$ are both continuous.
\end{lemma}
\begin{proof}
	Let us start with $F_{11}$. From \eqref{2-1}, it is obvious that $F_{11}: \left(-\frac{1}{2}, \frac{1}{2}\right)\times \mathbb{R} \times V^r \rightarrow X^k$ is continuous. We can rewrite
	$F_{11}$ as
	\begin{equation}\label{2-10}
		F_{11}=\Omega \left(d\sin(x)+\varepsilon\mathcal{R}_{11}(\varepsilon,f,g)\right),
	\end{equation}
	where $\mathcal{R}_{11}(\varepsilon,f): \left(-\frac{1}{2}, \frac{1}{2}\right)\times \mathbb{R} \times V^r \rightarrow X^k$ is continuous.
	
	The next step is to study $F_{12}$. To begin with, we show $F_{12}$ is odd. Notice that $f(x)\in X^{k+1}$, $g(x)\in X^k$ are both even functions and $f'(x)\in Y^{k-1}$ is odd. By substituting $-y$ for $y$ in $F_{12}$, we can deduce that $F_{12}(\varepsilon,f,g)$ is odd with respect to $x$.
	
	Then we are going to prove that the range of $F_{12}$ is in $Y^k$. Since $R(x)=1+\varepsilon f(x)$ and $\gamma(x)=1+\varepsilon g(x)$, the possible singularity for $\varepsilon=0$ may occur only when we take zeroth derivative of $F_{12}$. More precisely, it may occur only in $F_{124}$. To prove this case can not happen, we use Taylor's formula \eqref{2-1} to do following decomposition
	\begin{equation}\label{2-11}
		\frac{1}{{ \varepsilon^2\left(f(x)-f(y)\right)^2+4(1+\varepsilon f(x))(1+\varepsilon f(y))\sin^2\left(\frac{x-y}{2}\right)}}=\frac{1}{4\sin^2\left(\frac{x-y}{2}\right)}+\varepsilon K_\varepsilon(x,y),
	\end{equation}
	where we let $A=4\sin^2\left(\frac{x-y}{2}\right)$ and $B$ the remaining terms in the denominator. As for this decomposition, an important fact is that $K_\varepsilon(x,y)$ is not singular at $x=y$. Since $\sin(\cdot)$ is odd, it holds
	\begin{equation*}
		F_{124}=\frac{1}{\varepsilon}\int\!\!\!\!\!\!\!\!\!\; {}-{} \frac{\sin(x-y)}{4\sin^2\left(\frac{x-y}{2}\right)}dy+\varepsilon\mathcal R_{124}(\varepsilon,f,g)
=\varepsilon\mathcal R_{124}(\varepsilon,f,g),
	\end{equation*}
	where $\mathcal R_{124}(\varepsilon,f,g)$ is not singular at $\varepsilon=0$.
	
	Make the change of variable $x-y \rightarrow y$, and take $k$-$th$ partial derivatives of $F_{12}$ with respect to $x$. Since we have $1-\cos(y)=2\sin\left(\frac{y}{2}\right)$, for $\partial^kF_{121}$ it holds
	\begin{equation*}
		\partial^kF_{121}=\frac{1}{2}\int\!\!\!\!\!\!\!\!\!\; {}-{}(1+\varepsilon f(x))\partial^{k+1}f(x)(1+\varepsilon g(x-y))dy+\frac{1}{2}\int\!\!\!\!\!\!\!\!\!\; {}-{} \varepsilon(1+\varepsilon f(x))f'(x)\partial^kg(x-y)dy+l.o.t.,
	\end{equation*}
where we have used decomposition \eqref{2-11} and $l.o.t.$ represents all lower order terms. Since $f(x)\in X^{k+1}$, $g(x)\in X^k$ and $k\ge3$, we have $\|\partial^if(x\pm a\mathbf{i})\|_{L^\infty}\le C \|f\|_{X^{k+1}}<\infty$ and $\|\partial^ig(x\pm a\mathbf{i})\|_{L^\infty}\le C \|g\|_{X^k}<\infty$ for $i=0,1,2$. Using H\"older inequality and mean value theorem, we can conclude that
	\begin{equation*}
		\|	\partial^kF_{121}(x\pm a\mathbf{i})\|_{L^2}\le C\|f\|_{X^{k+1}}+C\|g\|_{X^k}<\infty.
	\end{equation*}

    We now turn to the most singular term $\partial^k F_{122}$, which has an expansion as
    \begin{equation*}
    	\begin{split}
    	\partial^k &F_{122}=\int\!\!\!\!\!\!\!\!\!\; {}-{} \frac{\varepsilon \partial^{k+1}f(x)(f(x)-f(x-y))\cos(y)(1+\varepsilon g(x-y))}{ \varepsilon^2\left(f(x)-f(x-y)\right)^2+4(1+\varepsilon f(x))(1+\varepsilon f(x-y))\sin^2\left(\frac{y}{2}\right)} dy\\	
    	&  \ \ \ \ \ \ +\int\!\!\!\!\!\!\!\!\!\; {}-{} \frac{\varepsilon f'(x)(\partial^kf(x)-\partial^kf(x-y))\cos(y)(1+\varepsilon g(x-y))}{ \varepsilon^2\left(f(x)-f(x-y)\right)^2+4(1+\varepsilon f(x))(1+\varepsilon f(x-y))\sin^2\left(\frac{y}{2}\right)} dy\\
    	& -\int\!\!\!\!\!\!\!\!\!\; {}-{} \frac{2\varepsilon^2 \partial^{k+1}f(x)(f(x)-f(x-y))\cos(y)(1+\varepsilon g(x-y))}{\left( \varepsilon^2\left(f(x)-f(x-y)\right)^2+4(1+\varepsilon f(x))(1+\varepsilon f(x-y))\sin^2\left(\frac{y}{2}\right)\right)^2}\times \bigg(\varepsilon(f(x)-f(x-y))\\
    	& \ \ \ \ \ \ \ (f'(x)-f'(x-y))+2\big((1+\varepsilon f(x))f'(x-y)+f'(x)(1+\varepsilon f(x-y))\big)\sin^2(\frac{y}{2})\bigg)dy\\
    	& \ \ \ \ \ \ +\int\!\!\!\!\!\!\!\!\!\; {}-{} \frac{\varepsilon f'(x)(f(x)-f(x-y))\cos(y)(1+\varepsilon \partial^kg(x-y))}{ \varepsilon^2\left(f(x)-f(x-y)\right)^2+4(1+\varepsilon f(x))(1+\varepsilon f(x-y))\sin^2\left(\frac{y}{2}\right)} dy\\	
    	& \ \ \ \ \ \ \ +l.o.t.\\
    	& \ \ \ \ \ \ =\partial F_{1221}+\partial F_{1222}+\partial F_{1223}+\partial F_{1224}+l.o.t..
        \end{split}
    \end{equation*}
    We first deal with $\partial F_{1221}$. Using the decomposition \eqref{2-11} with $\|K_\varepsilon(x,y)\|_{L^\infty}\le C(\varepsilon,\|f\|_{x^{k+1}})<\infty$ and the fact $|y|^2$ is the equivalent infinitesimal for $4\sin^2\left(\frac{y}{2}\right)$ as $|y|\to 0$, it holds
    \begin{equation*}
    	\begin{split}
    		\|\partial F_{1221}(x&\pm a\mathbf{i})\|_{L^2}\le C\left\|\int\!\!\!\!\!\!\!\!\!\; {}-{} \frac{\varepsilon \partial^{k+1}f(x\pm a\mathbf{i})(f(x\pm a\mathbf{i})-f(x\pm a\mathbf{i}-y))\cos(y)(1+\varepsilon g(x\pm a\mathbf{i}-y))}{4\sin^2\left(\frac{y}{2}\right)}dy\right\|_{L^2}\\
    		&\le C\|f\|_{X^{k+1}}(1+\|g (x\pm a\mathbf{i})\|_{L^\infty})\left|\int\!\!\!\!\!\!\!\!\!\; {}-{} \frac{f(x\pm a\mathbf{i})-f(x\pm a\mathbf{i}-y)dy}{4\sin^2\left(\frac{y}{2}\right)}\right|\\
    		&\le C\|f\|_{X^{k+1}}(1+\|g\|_{X^k})\left|\int\!\!\!\!\!\!\!\!\!\; {}-{} \frac{f'(x\pm a\mathbf{i})+f''(x\pm a\mathbf{i})\cdot|y|}{2\sin\left(\frac{y}{2}\right)}dy\right|\\
    		&\le C\|f\|_{X^{k+1}}(1+\|g\|_{X^k})\|_{L^\infty})\int\!\!\!\!\!\!\!\!\!\; {}-{} f''(x\pm a\mathbf{i})dy\\
    		&\le C\|f\|_{X^{k+1}}^2(1+\|g\|_{X^k}),
    	\end{split}
    \end{equation*}
    where we have used the identity $\int\!\!\!\!\!\!\!\!\!\; {}-{} \frac{1}{\sin\left(\frac{y}{2}\right)}dy=0$. For the term $\partial F_{1222}$, one can use a similar method to obtain the upper bound
    \begin{equation*}
    	\|\partial F_{1222}(x\pm a\mathbf{i})\|_{L^2}\le C\|f\|_{X^{k+1}}^2(1+\|g\|_{X^k}).
    \end{equation*}
    Hence in the next we will focus on $\partial F_{1223}$. By H\"{o}lder inequality and mean value theorem, we deduce that
     \begin{equation*}
    	\begin{split}
    		\|\partial F_{1223}(x&\pm a\mathbf{i})\|_{L^2}\le C\|f\|_{X^{k+1}}(1+\|g\|_{X^k})\left|\int\!\!\!\!\!\!\!\!\!\; {}-{} \frac{f(x\pm a\mathbf{i})-f(x\pm a\mathbf{i}-y)dy}{4\sin^2\left(\frac{y}{2}\right)}\right| \\
    		& \ \ \ \times\big(\|f'(x\pm a\mathbf{i})\|_{L^\infty}\|f''(x\pm a\mathbf{i})\|_{L^\infty}+(1+\|f(x\pm a\mathbf{i})\|_{L^\infty})\|f'(x\pm a\mathbf{i})\|_{L^\infty}\big)\\
    		&\le C\|f\|_{X^{k+1}}^3(1+\|f\|_{X^{k+1}})(1+\|g\|_{X^k})
    	\end{split}
    \end{equation*}
    through a direct calculation. In the last, it is easy to obtain
     \begin{equation*}
    	\|\partial F_{1224}(x\pm a\mathbf{i})\|_{L^2}\le C\|f\|_{X^{k+1}}^2(1+\|g\|_{X^k}).
    \end{equation*}
    Since $F_{122}$ is the most singular term in $F_{12}$, it always holds $\|F_{12}\|_{Y^k}\le C\|F_{122}\|_{Y^k}\le\infty$ from above discussion, and we can conclude that the range of $F_{12}$ is in $Y^k$.

    To prove the continuity of $F_{12}$, we also consider the most singular term $F_{122}$. We will use following notations: for a general function $u$, we denote
    $$\Delta u=u(x)-u(y), \ \ \ u=u(x), \ \ \ \tilde u=u(y),$$
    and
    $$D(u)=\varepsilon^2\Delta u^2+4(1+\varepsilon u)(1+\varepsilon\tilde u)\sin^2(\frac{x-y}{2}).$$
    To show the continuity with respect to $f$, letting $(f_1,g),(f_2,g)\in V^r$, then we can calculate the difference
    \begin{equation*}
    	\begin{split}
    		F_{122}(\varepsilon, f_1,g)&-F_{122}(\varepsilon, f_2,g)=(f_1'-f_2')\int\!\!\!\!\!\!\!\!\!\; {}-{}\frac{\varepsilon\Delta f_1\cos(x-y)(1+\varepsilon\tilde g)dy}{D(f_1)}\\
    		&+f_2'\int\!\!\!\!\!\!\!\!\!\; {}-{}\frac{\varepsilon(\Delta f_1-\Delta f_2)\cos(x-y)(1+\varepsilon\tilde g)dy}{D(f_1)}\\
    		&+\left(\int\!\!\!\!\!\!\!\!\!\; {}-{}\frac{f_2'\Delta f_2\cos(x-y)(1+\varepsilon\tilde g)dy}{D(f_1)}-\int\!\!\!\!\!\!\!\!\!\; {}-{}\frac{f_2'\Delta f_2\cos(x-y)(1+\varepsilon\tilde g)dy}{D(f_2)}\right)\\
    		&:=I_1+I_2+I_3.
    	\end{split}
    \end{equation*}
	For the first two terms $I_1$ and $I_2$, by the technique we have used before it is easy to prove
$$\|I_1\|_{Y^k}\le C\|f_1-f_2\|_{X^{k+1}}$$
 and
 $$\|I_{2}\|_{Y^k}\le C\|f_1-f_2\|_{X^{k+1}}.$$
  For the last term $I_3$, since
	\begin{equation}\label{2-12}
		\begin{split}
			\frac{1}{D(f_1)}&-\frac{1}{D(f_2)}\\
			&=\frac{\varepsilon^2(\Delta f_2^2-\Delta f_1^2)+4\varepsilon\big((f_2-f_1)(1+\varepsilon\tilde f_2)+(\tilde f_2-\tilde f_1)(1+\varepsilon f_1)\big)\sin^2(\frac{x-y}{2})}{D(f_1)D(f_2)},
		\end{split}
	\end{equation}
    it holds
    \begin{equation*}
    	\begin{split}
    	\|\partial^kI_3(x\pm a\mathbf{i})\|_{L^2}&\le C\|f_2\|_{X^{k+1}}^2\big(\|f_1\|_{X^{k+1}}+\|f_2\|_{X^{k+1}}+C\big)\|f_1-f_2\|_{X^{k+1}}(1+\|g\|_{X^k})\\
    	&\le C\|f_1-f_2\|_{X^{k+1}}.
    	\end{split}
    \end{equation*}
    On the other hand, using the fact that $F_{122}$ is linear with respect to $g$, we can verify the continuity of $F_{122}$ on $g$ directly. Since $F_{122}$ is the most singular term in $F_{12}$, we have actually shown that $F_{12}(\varepsilon, f, g): \left(-\frac{1}{2}, \frac{1}{2}\right) \times V^r \rightarrow Y^k$ is continuous. For further use, we write $F_{12}$ as
    \begin{equation}\label{2-13}
    	F_{12}=\frac{1}{2}f'(x)+\int\!\!\!\!\!\!\!\!\!\; {}-{}\frac{g(y)\sin(x-y)}{4\sin^2\left(\frac{x-y}{2}\right)}dy+\varepsilon \mathcal R_{12}(\varepsilon,f,g),
    \end{equation}
	where $\mathcal{R}_{12}(\varepsilon,f,g): \left(-\frac{1}{2}, \frac{1}{2}\right)\times \mathbb{R} \times V^r \rightarrow X^k$ is regular.
	
	Having dealt with $F_{12}$, we are to study $F_{13}$ and $F_{14}$. Notice that for $\boldsymbol{z}\in \Gamma^\varepsilon$ and $i\ge 1$, $\left| \left(\boldsymbol{z}(x)-(d,0)\right)-Q_{\frac{2\pi i}{m}}\left(\boldsymbol{z}(y)-(d,0)\right)\right|$ has positive lower bounds. So $F_{13}$ and $F_{14}$ are less singular than $F_{12}$, and $F_{14}(\varepsilon, f,g),F_{14}(\varepsilon, f,g): \left(-\frac{1}{2}, \frac{1}{2}\right)\times V^r \rightarrow Y^k$ are both continuous. Using the Taylor's formula \eqref{2-1}, we also rewrite $F_{13}$ and $F_{14}$ as
	\begin{equation}\label{2-14}
		F_{13}=\varepsilon \mathcal R_{13}(\varepsilon,f,g),
	\end{equation}
	and
	\begin{equation}\label{2-15}
		\begin{split}
		F_{14}&=-\sum_{i=1}^{m-1} \frac{(1-\cos(\frac{2\pi i}{m}))\sin (x)}{\left((-1+\cos(\frac{2\pi i}{m}))^2 +\sin^2(\frac{2\pi i}{m})\right)d}+\varepsilon\mathcal{R}_{14}(\varepsilon,f,g)\\
		&=-\frac{m-1}{2d}\sin(x)+\varepsilon\mathcal{R}_{14}(\varepsilon,f,g)
		\end{split}
	\end{equation}
    with both $\mathcal R_{13}$, $\mathcal R_{14}$ being good terms.
	
	The continuity for $F_2(\varepsilon, \Omega, f, g): \left(-\frac{1}{2}, \frac{1}{2}\right)\times \mathbb{R} \times V^r \rightarrow X^k$ can be verified in a similar way. Notice that projection operator $I-\mathcal P$ (defined in \eqref{1-10}) eliminates all constant terms in $\tilde F_2$, and hence there is no singularity in $F_2=(I-\mathcal P)\tilde F_2$. According to \eqref{2-1}, the four terms in $F_2$ can be rewritten into the following form respectively
	\begin{equation}\label{2-16}
		F_{21}=\Omega(d\cos{x}+\mathcal R_{21}(\varepsilon,f,g)),
	\end{equation}
    \begin{equation}\label{2-17}
    	F_{22}=f(x)-\int\!\!\!\!\!\!\!\!\!\; {}-{}\frac{f(x)-f(y)}{4\sin^2\left(\frac{x-y}{2}\right)}dy-\frac{1}{2}g(x)+\varepsilon \mathcal R_{22}(\varepsilon,f,g),
    \end{equation}
	\begin{equation}\label{2-18}
		F_{23}=\varepsilon \mathcal R_{23}(\varepsilon,f,g),
	\end{equation}
	\begin{equation}\label{2-19}
		F_{24}=-\frac{m-1}{2d}\cos(x)+\varepsilon\mathcal{R}_{24}(\varepsilon,f,g),
	\end{equation}
    where $\mathcal R_{2i}$ ($1\le i\le 4$) are all continuous from $ \left(-\frac{1}{2}, \frac{1}{2}\right)\times V^r$ to $X^k$.
	
	Summing up the above, we complete the proof of this lemma.
\end{proof}

For $(\varepsilon,\Omega,f,g)\in \left(-\frac{1}{2}, \frac{1}{2}\right)\times \mathbb{R} \times V^r$ and $(h_1,h_2)\in X^{k+1}\times X^k$, let
\begin{equation*}
	\partial_fF_i(\varepsilon, \Omega, f,g)h_1:=\lim\limits_{t\to0}\frac{1}{t}\left(F_i(\varepsilon, \Omega, f+th_1,g)-F_i(\varepsilon, \Omega, f,g)\right)
\end{equation*}
and
\begin{equation*}
	\partial_gF_i(\varepsilon, \Omega, f,g)h_2:=\lim\limits_{t\to0}\frac{1}{t}\left(F_i(\varepsilon, \Omega, f,g+th_2)-F_i(\varepsilon, \Omega, f,g)\right)
\end{equation*}
be the Gateaux derivatives of $F_i(\varepsilon, \Omega, f,g)$ with $i=1,2$. Then we have following lemma:
\begin{lemma}\label{lem2-2}
	For each $(\varepsilon,\Omega,f,g)\in \left(-\frac{1}{2}, \frac{1}{2}\right)\times \mathbb{R} \times V^r$ and $(h_1,h_2)\in X^{k+1}\times X^k$,
	$$\partial_fF_1(\varepsilon, \Omega, f,g)h_1: \ X^{k+1}\to Y^k, \ \ \ \partial_gF_1(\varepsilon, \Omega, f,g)h_2:  \ X^k\to Y^k$$
	and
	$$\partial_fF_2(\varepsilon, \Omega, f,g)h_1: \ X^{k+1}\to X^k, \ \ \ \partial_gF_2(\varepsilon, \Omega, f,g)h_2: \ X^k\to X^k$$
	are continuous.
\end{lemma}
\begin{proof}
	We only verify that $\partial_fF_1(\varepsilon, \Omega, f,g)h_1: \ X^{k+1}\to Y^k$ is continuous, since the continuity of other three Gateaux derivatives can be shown in the same way.
	
	According to Lemma \ref{lem2-1}, it is obvious that
	\begin{equation*}
		\partial_f F_{11}(\varepsilon,\Omega, f)h=\Omega\varepsilon|\varepsilon|\partial_f\mathcal{R}_1(\varepsilon,f)h
	\end{equation*}
	is continuous, where $\mathcal R_{11}(\varepsilon,f)$ is given in \eqref{2-4}.
	
	Then we claim $\partial_f F_{12}(\varepsilon, f,g)h_1=\partial_fF_{121}+\partial_fF_{122}+\partial_fF_{123}+\partial_fF_{124}$ is continuous (one can observe that $\partial_f F_{12}(\varepsilon, f,g)h_1$ is the main source of $\partial_f F_1(\varepsilon, f,g)h_1$ as $\varepsilon\to 0$), where
	\begin{equation*}
		\begin{split}
			\partial_f&F_{121}=\int\!\!\!\!\!\!\!\!\!\; {}-{} \frac{\varepsilon h_1(x)f'(x)(1-\cos(x-y))(1+\varepsilon g(y))}{ \varepsilon^2\left(f(x)-f(y)\right)^2+4(1+\varepsilon f(x))(1+\varepsilon f(y))\sin^2\left(\frac{x-y}{2}\right)} dy\\	
			&  \ \ \ \ \ \ \ \ \ +\int\!\!\!\!\!\!\!\!\!\; {}-{} \frac{(1+\varepsilon f(x))h_1'(x)(1-\cos(x-y))(1+\varepsilon g(y))}{ \varepsilon^2\left(f(x)-f(y)\right)^2+4(1+\varepsilon f(x))(1+\varepsilon f(y))\sin^2\left(\frac{x-y}{2}\right)} dy\\
			&  \ \ \ \ \ \ \ \ \  -\int\!\!\!\!\!\!\!\!\!\; {}-{} \frac{2\varepsilon(1+\varepsilon f(x))f'(x)(1-\cos(x-y))(1+\varepsilon g(y))}{\left( \varepsilon^2\left(f(x)-f(y)\right)^2+4(1+\varepsilon f(x))(1+\varepsilon f(y))\sin^2\left(\frac{x-y}{2}\right)\right)^2}\\
			& \times \left(\varepsilon(f(x)-f(y))(h_1(x)-h_1(y))+2\left((1+\varepsilon f(x))h_1(y)+h_1(x)(1+\varepsilon f(y))\right)\sin^2(\frac{x-y}{2})\right)dy,
		\end{split}
	\end{equation*}	
	\begin{equation*}
		\begin{split}
			\partial_f&F_{122}=\int\!\!\!\!\!\!\!\!\!\; {}-{} \frac{\varepsilon h_1'(x)(f(x)-f(y))\cos(x-y)(1+\varepsilon g(y))}{ \varepsilon^2\left(f(x)-f(y)\right)^2+4(1+\varepsilon f(x))(1+\varepsilon f(y))\sin^2\left(\frac{x-y}{2}\right)} dy\\	
			&  \ \ \ \ \ \ \ \ \ +\int\!\!\!\!\!\!\!\!\!\; {}-{} \frac{\varepsilon f'(x)(h_1(x)-h_1(y))\cos(x-y)(1+\varepsilon g(y))}{ \varepsilon^2\left(f(x)-f(y)\right)^2+4(1+\varepsilon f(x))(1+\varepsilon f(y))\sin^2\left(\frac{x-y}{2}\right)} dy\\
			&  \ \ \ \ \ \ \ \ \  -\int\!\!\!\!\!\!\!\!\!\; {}-{} \frac{2\varepsilon^2 f'(x)(f(x)-f(y))\cos(x-y)(1+\varepsilon g(y))}{\left( \varepsilon^2\left(f(x)-f(y)\right)^2+4(1+\varepsilon f(x))(1+\varepsilon f(y))\sin^2\left(\frac{x-y}{2}\right)\right)^2}\\
			& \times \left(\varepsilon(f(x)-f(y))(h_1(x)-h_1(y))+2\left((1+\varepsilon f(x))h_1(y)+h_1(x)(1+\varepsilon f(y))\right)\sin^2(\frac{x-y}{2})\right)dy,
		\end{split}
	\end{equation*}	
	\begin{equation*}
		\begin{split}
			\partial_f&F_{123}=-\int\!\!\!\!\!\!\!\!\!\; {}-{} \frac{\varepsilon h_1(x)(f(x)-f(y))\sin(x-y)(1+\varepsilon g(y))}{ \varepsilon^2\left(f(x)-f(y)\right)^2+4(1+\varepsilon f(x))(1+\varepsilon f(y))\sin^2\left(\frac{x-y}{2}\right)} dy\\	
			&  \ \ \ \ \ \ \ \ \ -\int\!\!\!\!\!\!\!\!\!\; {}-{} \frac{(1+\varepsilon f(x))(h_1(x)-h_1(y))\sin(x-y)(1+\varepsilon g(y))}{ \varepsilon^2\left(f(x)-f(y)\right)^2+4(1+\varepsilon f(x))(1+\varepsilon f(y))\sin^2\left(\frac{x-y}{2}\right)} dy\\
			&  \ \ \ \ \ \ \ \ \  +\int\!\!\!\!\!\!\!\!\!\; {}-{} \frac{2\varepsilon(1+\varepsilon f(x))(f(x)-f(y))\sin(x-y)(1+\varepsilon g(y))}{\left( \varepsilon^2\left(f(x)-f(y)\right)^2+4(1+\varepsilon f(x))(1+\varepsilon f(y))\sin^2\left(\frac{x-y}{2}\right)\right)^2}\\
			& \times \left(\varepsilon(f(x)-f(y))(h_1(x)-h_1(y))+2\left((1+\varepsilon f(x))h_1(y)+h_1(x)(1+\varepsilon f(y))\right)\sin^2(\frac{x-y}{2})\right)dy,
		\end{split}
	\end{equation*}
	\begin{equation*}
		\begin{split}
			\partial_f&F_{124}=\int\!\!\!\!\!\!\!\!\!\; {}-{} \frac{2(1+\varepsilon f(x))h_1(x)\sin(x-y)(1+\varepsilon g(y))}{ \varepsilon^2\left(f(x)-f(y)\right)^2+4(1+\varepsilon f(x))(1+\varepsilon f(y))\sin^2\left(\frac{x-y}{2}\right)} dy\\	
			&  \ \ \ \ \ \ \ \ \  -\int\!\!\!\!\!\!\!\!\!\; {}-{} \frac{2(1+\varepsilon f(x))^2\sin(x-y)(1+\varepsilon g(y))}{\left( \varepsilon^2\left(f(x)-f(y)\right)^2+4(1+\varepsilon f(x))(1+\varepsilon f(y))\sin^2\left(\frac{x-y}{2}\right)\right)^2}\\
			& \times \left(\varepsilon(f(x)-f(y))(h_1(x)-h_1(y))+2\left((1+\varepsilon f(x))h_1(y)+h_1(x)(1+\varepsilon f(y))\right)\sin^2(\frac{x-y}{2})\right)dy.
		\end{split}
	\end{equation*}
	First we are to show that these Gateaux derivatives are well-defined, that is
	$$\lim\limits_{t\to0}\left\|\frac{F_{12i}(\varepsilon, f+th_1,g)-F_{12i}(\varepsilon, f,g )}{t}-\partial_fF_{12i}\right\|_{Y^k}= 0, \ \ \text{for} \ \  i=1,2,3,4.$$
	To simplify our proof, we will deal with the most singular case $i=2$. Using the notations given in Lemma \ref{lem2-1}, we deduce that
	\begin{equation*}
		\begin{split}
			&\frac{F_{122}(\varepsilon, f+th_1,g)-F_{122}(\varepsilon, f,g )}{t}-\partial_fF_{122}\\
			&=\varepsilon h_1'(x)\int\!\!\!\!\!\!\!\!\!\; {}-{}(f(x)-f(y))\cos(x-y)(1+\varepsilon g(y))\bigg(\frac{1}{D(f+th_1)}-\frac{1}{D(f)}\bigg)dy\\
			& \ \ \ +\frac{\varepsilon f'(x)}{t}\int\!\!\!\!\!\!\!\!\!\; {}-{}(f(x)-f(y))\cos(x-y)(1+\varepsilon g(y))\\
			& \ \ \ \ \ \ \times\bigg(\frac{1}{D(f+th_1)}-\frac{1}{D(f)}+t\frac{2\varepsilon^2\Delta f\Delta h_1+4\big((1+\varepsilon\tilde f)\varepsilon h_1+\varepsilon \tilde h_1(1+\varepsilon f)\big)\sin^2(\frac{x-y}{2})}{D(f)^2}\bigg)dy\\
			& \ \ \ +\varepsilon f'(x)\int\!\!\!\!\!\!\!\!\!\; {}-{}(h_1(x)-h_1(y))\cos(x-y)(1+\varepsilon g(y))\bigg(\frac{1}{D(f+th_1)}-\frac{1}{D(f)}\bigg)dy\\
			& \ \ \ +t\varepsilon h_1'(x)\int\!\!\!\!\!\!\!\!\!\; {}-{}\frac{(h_1(x)-h_1(y))\cos(x-y)(1+\varepsilon g(y))}{D(f+th_1)}dy\\
			&=\partial_f^2F_{1221}+\partial_f^2F_{1222}+\partial_f^2F_{1223}+\partial_f^2F_{1224}.
		\end{split}
	\end{equation*}
    From \eqref{2-12}, for $\partial_f^2F_{1221}$ and $\partial_f^2F_{1223}$ we have
    \begin{equation*}
    	\|\partial^k\partial_f^2F_{1221}(x\pm a\mathbf{i})\|_{L^2}\le tC\|h_1\|_{x^{k+1}}\|f\|_{X^{k+1}}(1+\|g\|_{X^k})(1+\|f\|_{X^{k+1}})\|h_1\|_{X^k},
    \end{equation*}
    and
    \begin{equation*}
    	\|\partial^k\partial_f^2F_{1223}(x\pm a\mathbf{i})\|_{L^2}\le tC\|f\|_{x^{k+1}}\|f\|_{X^{k+1}}(1+\|g\|_{X^k})(1+\|f\|_{X^{k+1}})\|h_1\|_{X^k}.
    \end{equation*}
    For $\partial_f^2F_{1224}$, the computation will be a little tedious. However, using mean value theorem, we are able to verify that
    \begin{equation*}
    	\frac{1}{D(f+th_1)}-\frac{1}{D(f)}+t\frac{2\varepsilon^2\Delta f\Delta h_1+4\big(\varepsilon\tilde R h_1+\varepsilon \tilde h_1R\big)\sin^2(\frac{x-y}{2})}{D(f)^2}\le \frac{t^2}{4\sin^2\left(\frac{x-y}{2}\right)} \xi(\varepsilon,f,h_1),
    \end{equation*}
    where $\|\xi(\varepsilon,f,h_1)(x\pm a\mathbf{i})\|_{L^\infty}\le C(\varepsilon,\|f\|_{x^{k+1}},\|h_1\|_{x^{k+1}})<\infty$. As a result, it holds
    \begin{equation*}
    	\|\partial^k\partial_f^2F_{1224}(x\pm a\mathbf{i})\|_{L^2}\le tC\|f\|_{x^{k+1}}\|f\|_{X^{k+1}}(1+\|g\|_{X^k})\|\xi(\varepsilon,f,h_1)(x\pm a\mathbf{i})\|_{L^\infty}.
    \end{equation*}
    For the fourth term $\partial_f^2F_{1224}$, it is obvious that
    \begin{equation*}
    	\|\partial^k\partial_f^2F_{1224}(x\pm a\mathbf{i})\|_{L^2}\le tC\|f\|_{x^{k+1}}\|h_1\|_{X^{k+1}}(1+\|g\|_{X^k}),
    \end{equation*}
    and we have achieved our first target. In the next step, the continuity of $\partial_f F_{12}(\varepsilon, f,g)h_1$ should be shown. One can proceed as the proof Lemma \ref{lem2-1} and hence we omit the proof.

    To finish our discussion on $F_1$, we have to verify the continuity of $\partial_f F_{13}(\varepsilon, f,g)h_1$ and $\partial_f F_{14}(\varepsilon, f,g)h_1$. For this purpose, a similar method can be applied as above. Moreover, due to the positive lower bound of denominator, the proof of this part is much simpler.

    In the last, noting that $F_1$ and $F_2$ are almost linear dependent on $g$, it is easy to compute their Gateaux derivatives with respect to $g$, and we leave them to our readers. On the other hand, since $\partial_f F_{22}(\varepsilon, f,g)h_1=\partial_fF_{221}+\partial_fF_{222}+\partial_fF_{223}+\partial_fF_{224}$ is a little bit complicated and is the main source of $\partial_f F_2(\varepsilon, f,g)h_1$, we also write the four terms down here for readers' convenience.
    \begin{equation*}
    	\begin{split}
    		\partial_f&F_{221}=-\frac{2\varepsilon(1+\varepsilon g(x)) f'(x)h_1'(x)}{\big(\varepsilon^2f'(x)^2+(1+\varepsilon f(x))^2\big)^2}\int\!\!\!\!\!\!\!\!\!\; {}-{} \frac{(1+\varepsilon f(x))(1+\varepsilon f(y))(\cos(x-y)-1)(1+\varepsilon g(y))}{ \varepsilon^2\left(f(x)-f(y)\right)^2+4(1+\varepsilon f(x))(1+\varepsilon f(y))\sin^2\left(\frac{x-y}{2}\right)} dy\\
    		&  \ \ \  -\frac{2(1+\varepsilon g(x))(1+\varepsilon f(x))h_1(x)}{\big(\varepsilon^2f'(x)^2+(1+\varepsilon f(x))^2\big)^2}\int\!\!\!\!\!\!\!\!\!\; {}-{} \frac{(1+\varepsilon f(x))(1+\varepsilon f(y))(\cos(x-y)-1)(1+\varepsilon g(y))}{ \varepsilon^2\left(f(x)-f(y)\right)^2+4(1+\varepsilon f(x))(1+\varepsilon f(y))\sin^2\left(\frac{x-y}{2}\right)} dy\\
    		&  \ \ \ +\frac{(1+\varepsilon g(x))}{\varepsilon^2f'(x)^2+(1+\varepsilon f(x))^2}\int\!\!\!\!\!\!\!\!\!\; {}-{} \frac{(1+\varepsilon f(x))h_1(y)(\cos(x-y)-1)(1+\varepsilon g(y))}{ \varepsilon^2\left(f(x)-f(y)\right)^2+4(1+\varepsilon f(x))(1+\varepsilon f(y))\sin^2\left(\frac{x-y}{2}\right)} dy\\	
    		&  \ \ \ +\frac{(1+\varepsilon g(x))}{\varepsilon^2f'(x)^2+(1+\varepsilon f(x))^2}\int\!\!\!\!\!\!\!\!\!\; {}-{} \frac{ h_1(x)(1+\varepsilon f(y))(\cos(x-y)-1)(1+\varepsilon g(y))}{ \varepsilon^2\left(f(x)-f(y)\right)^2+4(1+\varepsilon f(x))(1+\varepsilon f(y))\sin^2\left(\frac{x-y}{2}\right)} dy\\
    		&  \ \ \  -\frac{(1+\varepsilon g(x))}{\varepsilon^2f'(x)^2+(1+\varepsilon f(x))^2}\int\!\!\!\!\!\!\!\!\!\; {}-{} \frac{2(1+\varepsilon f(x))(1+\varepsilon f(y))(\cos(x-y)-1)(1+\varepsilon g(y))}{\left( \varepsilon^2\left(f(x)-f(y)\right)^2+4(1+\varepsilon f(x))(1+\varepsilon f(y))\sin^2\left(\frac{x-y}{2}\right)\right)^2}\\
    		& \times \left(\varepsilon(f(x)-f(y))(h_1(x)-h_1(y))+2\left((1+\varepsilon f(x))h_1(y)+h_1(x)(1+\varepsilon f(y))\right)\sin^2(\frac{x-y}{2})\right)dy,
    	\end{split}
    \end{equation*}	
    \begin{equation*}
    	\begin{split}
    		\partial_f&F_{222}=\frac{2\varepsilon(1+\varepsilon g(x)) f'(x)h_1'(x)}{\big(\varepsilon^2f'(x)^2+(1+\varepsilon f(x))^2\big)^2}\int\!\!\!\!\!\!\!\!\!\; {}-{} \frac{\varepsilon^2 f'(x)(f(x)-f(y))\sin(x-y)(1+\varepsilon g(y))}{\left( \varepsilon^2\left(f(x)-f(y)\right)^2+4(1+\varepsilon f(x))(1+\varepsilon f(y))\sin^2\left(\frac{x-y}{2}\right)\right)^2}dy\\
    		& \ \ \ +\frac{2(1+\varepsilon g(x))(1+\varepsilon f(x))h_1(x)}{\big(\varepsilon^2f'(x)^2+(1+\varepsilon f(x))^2\big)^2}\int\!\!\!\!\!\!\!\!\!\; {}-{} \frac{\varepsilon^2 f'(x)(f(x)-f(y))\sin(x-y)(1+\varepsilon g(y))}{\left( \varepsilon^2\left(f(x)-f(y)\right)^2+4(1+\varepsilon f(x))(1+\varepsilon f(y))\sin^2\left(\frac{x-y}{2}\right)\right)^2}dy\\
    		& \ \ \ -\frac{(1+\varepsilon g(x))}{\varepsilon^2f'(x)^2+(1+\varepsilon f(x))^2}\int\!\!\!\!\!\!\!\!\!\; {}-{} \frac{\varepsilon^2 h_1'(x)(f(x)-f(y))\sin(x-y)(1+\varepsilon g(y))}{ \varepsilon^2\left(f(x)-f(y)\right)^2+4(1+\varepsilon f(x))(1+\varepsilon f(y))\sin^2\left(\frac{x-y}{2}\right)} dy\\	
    		&  \ \ \  -\frac{(1+\varepsilon g(x))}{\varepsilon^2f'(x)^2+(1+\varepsilon f(x))^2}\int\!\!\!\!\!\!\!\!\!\; {}-{} \frac{\varepsilon^2 f'(x)(h_1(x)-h_1(y))\sin(x-y)(1+\varepsilon g(y))}{ \varepsilon^2\left(f(x)-f(y)\right)^2+4(1+\varepsilon f(x))(1+\varepsilon f(y))\sin^2\left(\frac{x-y}{2}\right)} dy\\
    		&  \ \ \ +\frac{(1+\varepsilon g(x))}{\varepsilon^2f'(x)^2+(1+\varepsilon f(x))^2}\int\!\!\!\!\!\!\!\!\!\; {}-{} \frac{2\varepsilon^2 f'(x)(f(x)-f(y))\sin(x-y)(1+\varepsilon g(y))}{\left( \varepsilon^2\left(f(x)-f(y)\right)^2+4(1+\varepsilon f(x))(1+\varepsilon f(y))\sin^2\left(\frac{x-y}{2}\right)\right)^2}\\
    		& \times \left(\varepsilon(f(x)-f(y))(h_1(x)-h_1(y))+2\left((1+\varepsilon f(x))h_1(y)+h_1(x)(1+\varepsilon f(y))\right)\sin^2(\frac{x-y}{2})\right)dy,
    	\end{split}
    \end{equation*}	
    \begin{equation*}
    	\begin{split}
    		\partial_f&F_{223}=\frac{2\varepsilon(1+\varepsilon g(x)) f'(x)h_1'(x)}{\big(\varepsilon^2f'(x)^2+(1+\varepsilon f(x))^2\big)^2}\int\!\!\!\!\!\!\!\!\!\; {}-{} \frac{\varepsilon(1+\varepsilon f(x))(f(x)-f(y))(1+\varepsilon g(y))}{\left( \varepsilon^2\left(f(x)-f(y)\right)^2+4(1+\varepsilon f(x))(1+\varepsilon f(y))\sin^2\left(\frac{x-y}{2}\right)\right)^2}dy\\
    		& \ \ \ +\frac{2(1+\varepsilon g(x))(1+\varepsilon f(x))h_1(x)}{\big(\varepsilon^2f'(x)^2+(1+\varepsilon f(x))^2\big)^2}\int\!\!\!\!\!\!\!\!\!\; {}-{} \frac{\varepsilon(1+\varepsilon f(x))(f(x)-f(y))(1+\varepsilon g(y))}{\left( \varepsilon^2\left(f(x)-f(y)\right)^2+4(1+\varepsilon f(x))(1+\varepsilon f(y))\sin^2\left(\frac{x-y}{2}\right)\right)^2}dy\\
    		& \ \ \  -\frac{(1+\varepsilon g(x))}{\varepsilon^2f'(x)^2+(1+\varepsilon f(x))^2}\int\!\!\!\!\!\!\!\!\!\; {}-{} \frac{\varepsilon h_1(x)(f(x)-f(y))(1+\varepsilon g(y))}{ \varepsilon^2\left(f(x)-f(y)\right)^2+4(1+\varepsilon f(x))(1+\varepsilon f(y))\sin^2\left(\frac{x-y}{2}\right)} dy\\	
    		&  \ \ \ -\frac{(1+\varepsilon g(x))}{\varepsilon^2f'(x)^2+(1+\varepsilon f(x))^2}\int\!\!\!\!\!\!\!\!\!\; {}-{} \frac{(1+\varepsilon f(x))(h_1(x)-h_1(y))(1+\varepsilon g(y))}{ \varepsilon^2\left(f(x)-f(y)\right)^2+4(1+\varepsilon f(x))(1+\varepsilon f(y))\sin^2\left(\frac{x-y}{2}\right)} dy\\
    		&  \ \ \  +\frac{(1+\varepsilon g(x))}{\varepsilon^2f'(x)^2+(1+\varepsilon f(x))^2}\int\!\!\!\!\!\!\!\!\!\; {}-{} \frac{2\varepsilon(1+\varepsilon f(x))(f(x)-f(y))(1+\varepsilon g(y))}{\left( \varepsilon^2\left(f(x)-f(y)\right)^2+4(1+\varepsilon f(x))(1+\varepsilon f(y))\sin^2\left(\frac{x-y}{2}\right)\right)^2}\\
    		& \times \left(\varepsilon(f(x)-f(y))(h_1(x)-h_1(y))+2\left((1+\varepsilon f(x))h_1(y)+h_1(x)(1+\varepsilon f(y))\right)\sin^2(\frac{x-y}{2})\right)dy,
    	\end{split}
    \end{equation*}
    \begin{equation*}
    	\begin{split}
    		\partial_f&F_{224}=-\frac{2\varepsilon(1+\varepsilon g(x)) f'(x)h_1'(x)}{\big(\varepsilon^2f'(x)^2+(1+\varepsilon f(x))^2\big)^2}\int\!\!\!\!\!\!\!\!\!\; {}-{} \frac{\varepsilon f'(x)(1+\varepsilon f(x))\sin(x-y)(1+\varepsilon g(y))}{\left( \varepsilon^2\left(f(x)-f(y)\right)^2+4(1+\varepsilon f(x))(1+\varepsilon f(y))\sin^2\left(\frac{x-y}{2}\right)\right)^2}dy\\
    		& \ \ \ -\frac{2(1+\varepsilon g(x))(1+\varepsilon f(x))h_1(x)}{\big(\varepsilon^2f'(x)^2+(1+\varepsilon f(x))^2\big)^2}\int\!\!\!\!\!\!\!\!\!\; {}-{} \frac{\varepsilon f'(x)(1+\varepsilon f(x))\sin(x-y)(1+\varepsilon g(y))}{\left( \varepsilon^2\left(f(x)-f(y)\right)^2+4(1+\varepsilon f(x))(1+\varepsilon f(y))\sin^2\left(\frac{x-y}{2}\right)\right)^2}dy\\
    		& \ \ \ +\frac{(1+\varepsilon g(x))}{\varepsilon^2f'(x)^2+(1+\varepsilon f(x))^2}\int\!\!\!\!\!\!\!\!\!\; {}-{} \frac{ h_1'(x)(1+\varepsilon f(x))\sin(x-y)(1+\varepsilon g(y))}{ \varepsilon^2\left(f(x)-f(y)\right)^2+4(1+\varepsilon f(x))(1+\varepsilon f(y))\sin^2\left(\frac{x-y}{2}\right)} dy\\	
    		& \ \ \ +\frac{(1+\varepsilon g(x))}{\varepsilon^2f'(x)^2+(1+\varepsilon f(x))^2}\int\!\!\!\!\!\!\!\!\!\; {}-{} \frac{\varepsilon f'(x)h_1(x)\sin(x-y)(1+\varepsilon g(y))}{ \varepsilon^2\left(f(x)-f(y)\right)^2+4(1+\varepsilon f(x))(1+\varepsilon f(y))\sin^2\left(\frac{x-y}{2}\right)} dy\\	
    	    & \ \ \ -\frac{(1+\varepsilon g(x))}{\varepsilon^2f'(x)^2+(1+\varepsilon f(x))^2}\int\!\!\!\!\!\!\!\!\!\; {}-{} \frac{2\varepsilon f'(x)(1+\varepsilon f(x))\sin(x-y)(1+\varepsilon g(y))}{\left( \varepsilon^2\left(f(x)-f(y)\right)^2+4(1+\varepsilon f(x))(1+\varepsilon f(y))\sin^2\left(\frac{x-y}{2}\right)\right)^2}\\
    		& \times \left(\varepsilon(f(x)-f(y))(h_1(x)-h_1(y))+2\left((1+\varepsilon f(x))h_1(y)+h_1(x)(1+\varepsilon f(y))\right)\sin^2(\frac{x-y}{2})\right)dy.
    	\end{split}
    \end{equation*}

    Concluding all the facts above, we have completed the proof.
\end{proof}

According to Lemma \ref{lem2-1} and Lemma \ref{lem2-2}, when $\varepsilon=0$ and $f,g\equiv 0$, the Gateaux derivatives are
\begin{equation*}
	\partial_fF_1(\varepsilon, \Omega, f,g)h_1=\frac{1}{2}h_1'(x),
\end{equation*}
\begin{equation*}
	\partial_gF_1(\varepsilon, \Omega, f,g)h_2=\int\!\!\!\!\!\!\!\!\!\; {}-{}\frac{h_2(y)\sin(x-y)}{4\sin^2\left(\frac{x-y}{2}\right)}dy,
\end{equation*}
\begin{equation*}
	\partial_fF_2(\varepsilon, \Omega, f,g)h_1=h_1(x)-\int\!\!\!\!\!\!\!\!\!\; {}-{}\frac{h_1(x)-h_1(y)}{4\sin^2\left(\frac{x-y}{2}\right)}dy,
\end{equation*}
\begin{equation*}
	\partial_gF_1(\varepsilon, \Omega, f,g)h_2=-\frac{1}{2}h_2(x).
\end{equation*}
Take $(h_1,h_2)\in X^{k+1}\times X^k$, where
\begin{equation}\label{2-20}
	h_1(x)=\sum_{j=1}^\infty a_j \cos(jx) \ \ \ \text{and} \ \ \ h_2(x)=\sum_{j=1}^\infty b_j \cos(jx).
\end{equation}
If we denote $\mathbf F=(F_1,F_2)(\varepsilon,\Omega,f,g)$, the linearization of $\mathbf F$ at $(0,\Omega,0,0)$ has the following Fourier series form
\begin{equation}\label{2-21}
	D \mathbf F(0,\Omega,0,0) (h_1, h_2)=\begin{pmatrix}
		\partial_fF_1 (0, \Omega, 0,0) h_1 + \partial_gF_1 (0, \Omega, 0,0) h_2\\
	    \partial_fF_2(0, \Omega, 0,0) h_1 + \partial_gF_2 (0, \Omega, 0,0) h_2
	\end{pmatrix}
	=\sum_{j=1}^\infty\begin{pmatrix}  \hat a_j \sin(jx) \\  \hat b_j \cos(jx)\end{pmatrix},
\end{equation}
where
\begin{equation*}
	 \begin{pmatrix} \hat a_j  \\ \hat b_j \end{pmatrix}=M_j  \begin{pmatrix} a_j  \\  b_j \end{pmatrix}
\end{equation*}
with $M_j$ a $2\times 2$ matrix and $a_j,b_j$ given in \eqref{2-20}. In the following lemma we will calculate the matrix $M_j$, and claim that the linearization of $\mathbf F(\varepsilon,\Omega,f)$ at $(0,\Omega,0,0)$ is an isomorphism from $\mathcal A^k$ to $\mathcal B^k$. Notice that $M_1$ is not invertible. This is the reason why we give $\mathcal A^k$ and $\mathcal B^k$ a special restriction on the first Fourier coefficient.
\begin{lemma}\label{lem2-3}
	One has
	$$M_j=\begin{pmatrix} -j/2 & 1/2 \\ (2-j)/2 & -1/2 \end{pmatrix}.$$
	Moreover, $D\mathbf F(0,\Omega,0,0) (h_1, h_2)$ is an isomorphism from $\mathcal A^k$ to $\mathcal B^k$.
\end{lemma}
\begin{proof}
	We first show
	\begin{equation}\label{2-22}
		\int\!\!\!\!\!\!\!\!\!\; {}-{}\frac{\cos(jx)\sin(x-y)}{4\sin^2\left(\frac{x-y}{2}\right)}dy=\frac{1}{2}\sin(jx)
	\end{equation}
    holds for all $j\ge1$ and $j\in \mathbb N^*$. It can be deduced from the identity
    \begin{equation*}
    	\int\!\!\!\!\!\!\!\!\!\; {}-{}\frac{\cos(jy)\sin(x-y)}{4\sin^2\left(\frac{x-y}{2}\right)}dy=\frac{1}{2}\int\!\!\!\!\!\!\!\!\!\; {}-{}\cos(jy)\cot\left(\frac{x-y}{2}\right)dy=\frac{1}{2} \mathcal H(\cos(jx))(x),
    \end{equation*}
    where $\mathcal H(\cdot)$ is the Hilbert transform on torus. This fact implies that
    \begin{equation*}
    	\partial_fF_1(\varepsilon, \Omega, f,g)h_1=-\sum_{j=1}^\infty\frac{j}{2} a_j \sin(jx),
    \end{equation*}
    and
	\begin{equation*}
		\partial_gF_1(\varepsilon, \Omega, f,g)h_2=\frac{1}{2}\sum_{j=1}^\infty b_j\sin(jx).
	\end{equation*}	
    On the other hand,  for all $j\ge1$ and $j\in \mathbb N^*$ it holds
    \begin{equation*}
    	\int\!\!\!\!\!\!\!\!\!\; {}-{}\frac{\cos(jx)-\cos(jy)}{4\sin^2\left(\frac{x-y}{2}\right)}dy=\frac{1}{2}(-\Delta)^{\frac{1}{2}}\cos(jx)=\frac{j}{2}\cos(jx).
    \end{equation*}
    Then we obtain
    \begin{equation*}
    	\partial_fF_2(\varepsilon, \Omega, f,g)h_1=\sum_{j=1}^\infty\frac{2-j}{2} a_j \sin(jx),
    \end{equation*}
    and
    \begin{equation*}
    	\partial_gF_2(\varepsilon, \Omega, f,g)h_2=-\frac{1}{2}\sum_{j=1}^\infty b_j\sin(jx).
    \end{equation*}	
    Thus we have verified the first part of this lemma.

    Now we are going to prove $D\mathbf F(0,\Omega,0,0) (h_1, h_2)$ is an isomorphism from $\mathcal A^k$ to $\mathcal B^k$. Recall the definition of $\mathcal A^k$ and $\mathcal B^k$ given in the beginning of this section. From Lemma \eqref{2-2} and $M_1=\begin{pmatrix} -1/2 & 1/2 \\ 1/2 & -1/2 \end{pmatrix}$, it is obvious that $D\mathbf{F}$ maps $\mathcal A^k$ into $\mathcal B^k$. Hence only the invertibility need to be considered.

    For $j\ge 2$, $\det(M_j)=\frac{j-1}{2}$ which implies $M_j$ is invertible at this time. We can compute the inverse of $M_j$ as
    \begin{equation}\label{2-23}
       M_j^{-1}=\begin{pmatrix} -1/(j-1) & -1/(j-1) \\ (j-2)/(j-1) & -j/(j-1) \end{pmatrix}, \ \ \ \forall \, j\ge 2.
    \end{equation}
    Thus for any $(u,v)\in \mathcal B^k$ with the formulation
    \begin{equation*}
    	u=\sum_{j=1}^\infty p_j\sin(jx) \ \ \ \text{and} \ \ \ g=-p_1\cos(x)+\sum_{j=1}^\infty q_j\cos(jx),
    \end{equation*}
    we can write $D\mathbf F^{-1}(0,\Omega,0,0)(u,v)$ as
    \begin{equation*}
    		D \mathbf F^{-1}(0,\Omega,0,0) (u,v)=\begin{pmatrix}  -p_1 \cos(jx) \\  p_1 \cos(jx)\end{pmatrix}+\sum_{j=2}^\infty\begin{pmatrix}  \tilde p_j \cos(jx) \\  \tilde q_j \cos(jx)\end{pmatrix},
    \end{equation*}
    where
    \begin{equation*}
    	\begin{pmatrix} \tilde p_j  \\ \tilde q_j \end{pmatrix}=M_j^{-1}  \begin{pmatrix} p_j  \\  q_j \end{pmatrix}, \ \ \ \forall \, j\ge 2.
    \end{equation*}
    From \eqref{2-23}, we see $\tilde p_j\sim -j^{-1}(p_j+q_j)$ and $\tilde q_j\sim p_j-q_j$ as $j\to +\infty$. Hence $D \mathbf F^{-1}(0,\Omega,0,0) (u,v)$ does belong to $\mathcal A^k$ and the proof is finished.
\end{proof}

One can easily verify that for
\begin{equation}\label{2-24}
	\Omega^*:=\frac{m-1}{2d^2},
\end{equation}
the four-tuple $(0,\Omega^*,0,0)$ is a trivial solution to $\mathbf F(\varepsilon,\Omega,f,g)=(F_1,F_2)(\varepsilon,\Omega,f,g)=0$, which corresponds to co-rotating $m$-fold point vortex solution. To apply the implicit function theorem and obtain concentrative vortex sheet solutions, we need to choose suitable $\Omega=(\varepsilon,f,g)$ such that $\mathbf F(\varepsilon,\Omega,f,g)$ maps $V^r\subset\mathcal A^k$ into $\mathcal B^k$. This is achieved in the following lemma.
\begin{lemma}\label{lem2-4}
	Let $\Omega^*$ be given by \eqref{2-24}.Then there exists some continuous function $\Omega_\varepsilon(\varepsilon, f,g):(-1/2,1/2)\times V^r\to \mathbb R$, such that if
	\begin{equation*}
		\Omega(\varepsilon,f,g):=\Omega^*+\varepsilon\cdot\Omega_\varepsilon(\varepsilon,f,g),
	\end{equation*}
then $\mathbf{F^*}(\varepsilon,f,g):=\mathbf F(\varepsilon,\Omega(\varepsilon,f,g),f,g)$ maps $(-1/2, 1/2)\times V^r$ into $\mathcal B^k$. Moreover, for any $(\varepsilon,f,g)\in (-1/2,1/2)\times V^r$, $\partial_f\Omega_\varepsilon(\varepsilon,f,g)(\cdot):X^{k+1}\to \mathbb R$ and  $\partial_g\Omega_\varepsilon(\varepsilon,f,g)(\cdot):X^k\to \mathbb R$ are both continuous.
\end{lemma}
\begin{proof}
	According to the definition of $\mathcal B^k$, it is sufficient to find $\Omega=\Omega(\varepsilon,f,g)$ such that
	\begin{equation}\label{2-25}
		\int\!\!\!\!\!\!\!\!\!\; {}-{} F_1(\varepsilon, \Omega, f,g)\sin(x) dx=-\int\!\!\!\!\!\!\!\!\!\; {}-{} F_2(\varepsilon, \Omega, f,g)\cos(x) dx
	\end{equation}
	for $\varepsilon$ sufficiently small and $(f,g)\in V^r$. Let
	 \begin{equation*}
		f=\sum_{j=1}^\infty a_j\cos(jx) \ \ \ \text{and} \ \ \ g=-a_1\cos(x)+\sum_{j=1}^\infty b_j\cos(jx).
	\end{equation*}
	Then by Lemma \ref{lem2-1}, Lemma \ref{lem2-3} and \eqref{2-24}, identity \eqref{2-25} is equivalent to
	\begin{equation*}
		\begin{split}
		\Omega d&-a_1+\int\!\!\!\!\!\!\!\!\!\; {}-{} \sum_{1=1}^4 \varepsilon \mathcal R_{1i}(\varepsilon,f,g)\sin(x) dx-\Omega^*d\\
		&=-\big(\Omega d+a_1-\int\!\!\!\!\!\!\!\!\!\; {}-{} \sum_{1=1}^4 \varepsilon \mathcal R_{2i}(\varepsilon,f,g)\cos(x) dx-\Omega^*d\big).
		\end{split}
	\end{equation*}
    So we deduce that
    \begin{equation*}
    	\Omega(\varepsilon,f,g)=\Omega^*-\frac{\varepsilon}{2d}\int\!\!\!\!\!\!\!\!\!\; {}-{} \sum_{1=1}^4 \mathcal R_{1i}(\varepsilon,f,g)\sin(x) dx+\frac{\varepsilon}{2d}\int\!\!\!\!\!\!\!\!\!\; {}-{} \sum_{1=1}^4 \mathcal R_{2i}(\varepsilon,f,g)\cos(x) dx.
    \end{equation*}
	Let
	\begin{equation*}
		\Omega_\varepsilon(\varepsilon,f,g):=-\frac{1}{2d}\int\!\!\!\!\!\!\!\!\!\; {}-{} \sum_{1=1}^4 \mathcal R_{1i}(\varepsilon,f,g)\sin(x) dx+\frac{1}{2d}\int\!\!\!\!\!\!\!\!\!\; {}-{} \sum_{1=1}^4 \mathcal R_{2i}(\varepsilon,f,g)\cos(x) dx.
	\end{equation*}
	From the discussion in Lemma \ref{lem2-2}, it is easy to see that $\partial_f\Omega_\varepsilon(\varepsilon,f,g)(\cdot):X^{k+1}\to \mathbb R$ and  $\partial_g\Omega_\varepsilon(\varepsilon,f,g)(\cdot):X^k\to \mathbb R$ are continuous. Thus the proof is therefore completed.	
\end{proof}

Having made all the preparations in the previous Lemmas, we are now able to prove Theorem \ref{thm1}.

{\bf Proof of Theorem \ref{thm1}:}
We first show that $D \mathbf{F^*}(\varepsilon,f,g)(\cdot, \cdot): \mathcal A^k \to \mathcal B^k$ is an isomorphism. By the chain rule,  we have for $i=1,2$
\begin{equation*}
	\partial_f F_i^*(0,0,0)h_1=\partial_\Omega F_i(0,\Omega^*,0,0)\partial_f\Omega(0,0,0)h_1+\partial_f F_1(0,\Omega^*,0,0)h_1, \ \ \ \forall \, h_1\in X^{k+1},
\end{equation*}
and
\begin{equation*}
	\partial_g F_i^*(0,0,0)h_2=\partial_\Omega F_i(0,\Omega^*,0,0)\partial_g\Omega(0,0,0)h_2+\partial_g F_1(0,\Omega^*,0,0)h_2, \ \ \ \ \forall \, h_2\in X^k.
\end{equation*}
On the other hand, using the formulation of angular velocity $\Omega$ in Lemma \ref{lem2-4}, we deduce that $\partial_f\Omega(0,0,0)h_1=\partial_{g}\Omega(0,0,0)h_2\equiv0$ for any $h_1\in X^{k+1}$ and $h_2\in X^k$. Hence for $i=1,2$, it should be
\begin{equation*}
\partial_f F_i^*(0,0,0)(\cdot)=\partial_f F_i(0,\Omega^*,0,0)(\cdot) \ \ \ \text{and} \ \ \ \partial_g F_i^*(0,0,0)(\cdot)=\partial_g F_i(0,\Omega^*,0,0)(\cdot),
\end{equation*}
and we have achieved our target by Lemma \ref{lem2-3}.

According to Lemma \ref{lem2-1} -- Lemma \ref{lem2-4}, we can apply implicit function theorem to derive that there exists $\varepsilon_0>0$ such that
\begin{equation*}
	\left\{(\varepsilon,f,g)\in [-\varepsilon_0,\varepsilon_0]\times V^r \ : \ \mathbf{F^*}(\varepsilon,f,g)=0\right\}
\end{equation*}
is parameterized by one-dimensional curve $\varepsilon\in [-\varepsilon_0,\varepsilon_0]\to (\varepsilon, f_\varepsilon,g_\varepsilon)$. Moreover, it is not difficult to verify that for $\varepsilon\in [-\varepsilon_0,\varepsilon_0]$,
\begin{equation*}
	-\frac{1}{2d}\int\!\!\!\!\!\!\!\!\!\; {}-{} \sum_{1=1}^4 \mathcal R_{1i}(\varepsilon,f_\varepsilon,g_\varepsilon)\sin(jx) dx+\frac{1}{2d}\int\!\!\!\!\!\!\!\!\!\; {}-{} \sum_{1=1}^4 \mathcal R_{2i}(\varepsilon,f_\varepsilon,g_\varepsilon)\cos(jx) dx \not\equiv 0, \ \ \ \exists \, j\in \mathbb N^*.
\end{equation*}
So we have actually obtained a family of nontrivial vortex sheet solutions.

To show the convexity of the area surrounded by the closed curve corresponding to the vortex sheet, we use direct computation. Indeed,
\begin{align*}
	\varepsilon \kappa(x)=\frac{R(x)^2+2R'(x)^2-R(x)R''(x)}{\left(R(x)^2+R'(x)^2\right)^{\frac{3}{2}}}=\frac{1+O(\varepsilon)}{1+O(\varepsilon)}>0,
\end{align*}
for $\varepsilon$ small and each $x\in[0,2\pi)$. Hence we claim that every component of vorticity is supported on a closed curve with convex interior. Since $\mathcal A^k$ and $\mathcal B^k$ are analytic spaces, this curve is also of $C^\infty$.

Let $\bar f(x)=f(-x)$ and $\bar g(x)=g(-x)$. The last thing to be verified is that if $(\varepsilon,f,g)$ is a solution to $\mathbf {F^*}(\varepsilon,f,g)=0$, then $(-\varepsilon, \bar f,\bar g)$ is also a solution. By replacing $y$ with $-y$ in the integral representation of $\mathbf F(\varepsilon,\Omega,f,g)$ and using the fact that $f,g$ are both even functions, we derive $\Omega_\varepsilon(\varepsilon,f)=\Omega_\varepsilon(-\varepsilon,\bar f,\bar g)$. Then we can insert this relationship into $\mathbf F(\varepsilon,\Omega,f,g)$ and obtain $\mathbf{F^*}(-\varepsilon, \bar f,\bar g)=0$ by a similar substitution of variables as before.

Thus we have finished the proof of Theorem \ref{thm1}.\qed

\section{Existence of traveling vortex sheets}
As before, let $\varepsilon\in (-\frac{1}{2},\frac{1}{2})$, $R(x)=1+\varepsilon f(x)$, $\gamma(x)=1+\varepsilon g(x)$
with $x\in [0,2\pi)$, and $f, g$ be two $C^1$ functions. By \eqref{1-15}, we need to find the zero points of following two functionals
\begin{equation*}
	G_1(\varepsilon,W,f,g)=G_{11}+G_{12}+G_{13}+G_{14},
\end{equation*}
where
\begin{equation*}
	G_{11}=-W\left((1+\varepsilon f(x))\sin(x)-\varepsilon f'(x)\cos(x)\right),
\end{equation*}

\begin{equation*}
	\begin{split}
		G_{12}&=\frac{1}{\varepsilon}\int\!\!\!\!\!\!\!\!\!\; {}-{} \frac{\varepsilon(1+\varepsilon f(x))f'(x)(1-\cos(x-y))(1+\varepsilon g(y))}{ \varepsilon^2\left(f(x)-f(y)\right)^2+4(1+\varepsilon f(x))(1+\varepsilon f(y))\sin^2\left(\frac{x-y}{2}\right)} dy\\
		& \ \ \ +\frac{1}{\varepsilon}\int\!\!\!\!\!\!\!\!\!\; {}-{} \frac{\varepsilon^2 f'(x)(f(x)-f(y))\cos(x-y)(1+\varepsilon g(y))}{ \varepsilon^2\left(f(x)-f(y)\right)^2+4(1+\varepsilon f(x))(1+\varepsilon f(y))\sin^2\left(\frac{x-y}{2}\right)}dy\\
		& \ \ \ -\frac{1}{\varepsilon}\int\!\!\!\!\!\!\!\!\!\; {}-{} \frac{\varepsilon(1+\varepsilon f(x))(f(x)-f(y))\sin(x-y)(1+\varepsilon g(y))}{ \varepsilon^2\left(f(x)-f(y)\right)^2+4(1+\varepsilon f(x))(1+\varepsilon f(y))\sin^2\left(\frac{x-y}{2}\right)}dy\\
		& \ \ \ +\frac{1}{\varepsilon}\int\!\!\!\!\!\!\!\!\!\; {}-{} \frac{(1+\varepsilon f(x))^2\sin(x-y)(1+\varepsilon g(y))}{ \varepsilon^2\left(f(x)-f(y)\right)^2+4(1+\varepsilon f(x))(1+\varepsilon f(y))\sin^2\left(\frac{x-y}{2}\right)}dy,\\
	\end{split}
\end{equation*}

\begin{equation*}
	\begin{split}
		G_{13}&=-\int\!\!\!\!\!\!\!\!\!\; {}-{} \frac{\varepsilon(1+\varepsilon f(x))f'(y)(1+\varepsilon g(y))}{\left|(\varepsilon R(x)\cos(x)+\varepsilon R(y)\cos(y)-2d)^2+(\varepsilon R(x)\sin(x)+\varepsilon R(y)\sin(y))^2\right|^2}dy\\
		&\ \ \ -\int\!\!\!\!\!\!\!\!\!\; {}-{} \frac{\varepsilon f'(x)(1+\varepsilon f(y))\cos(x-y)(1+\varepsilon g(y))}{\left|(\varepsilon R(x)\cos(x)+\varepsilon R(y)\cos(y)-2d)^2+(\varepsilon R(x)\sin(x)+\varepsilon R(y)\sin(y))^2\right|^2}dy\\
		& \ \ \ +\int\!\!\!\!\!\!\!\!\!\; {}-{} \frac{(1+\varepsilon f(x))(1+\varepsilon f(x))\sin(x-y)(1+\varepsilon g(y))}{\left|(\varepsilon R(x)\cos(x)+\varepsilon R(y)\cos(y)-2d)^2+(\varepsilon R(x)\sin(x)+\varepsilon R(y)\sin(y))^2\right|^2}dy,
	\end{split}
\end{equation*}

\begin{equation*}
	\begin{split}
		G_{14}&=-\int\!\!\!\!\!\!\!\!\!\; {}-{} \frac{2d(1+\varepsilon f(x))\sin(x)(1+\varepsilon g(y))}{\left|(\varepsilon R(x)\cos(x)+\varepsilon R(y)\cos(y)-2d)^2+(\varepsilon R(x)\sin(x)+\varepsilon R(y)\sin(y))^2\right|^2}dy\\
		& \ \ \ +\int\!\!\!\!\!\!\!\!\!\; {}-{} \frac{2\varepsilon df'(x)\cos(x)(1+\varepsilon g(y))}{\left|(\varepsilon R(x)\cos(x)+\varepsilon R(y)\cos(y)-2d)^2+(\varepsilon R(x)\sin(x)+\varepsilon R(y)\sin(y))^2\right|^2}dy,
	\end{split}
\end{equation*}
and
\begin{equation*}
	G_2(\varepsilon,W,f,g)=(I-\mathcal P)\tilde G_2=(I-\mathcal P)(\tilde G_{21}+\tilde G_{22}+\tilde G_{23}+\tilde G_{24}),
\end{equation*}
where
\begin{equation*}
	\tilde G_{21}=-\frac{W(1+\varepsilon g(x))}{\varepsilon^2f'(x)^2+(1+\varepsilon f(x))^2}\cdot\big((1+\varepsilon f(x))\cos(x)+\varepsilon f'(x)\sin(x)\big),
\end{equation*}

\begin{equation*}
	\begin{split}
		\tilde G_{22}&=\frac{1}{\varepsilon}\cdot\frac{(1+\varepsilon g(x))}{\varepsilon^2f'(x)^2+(1+\varepsilon f(x))^2}\int\!\!\!\!\!\!\!\!\!\; {}-{} \frac{(1+\varepsilon f(x))(1+\varepsilon f(y))(\cos(x-y)-1)(1+\varepsilon g(y))}{ \varepsilon^2\left(f(x)-f(y)\right)^2+4(1+\varepsilon f(x))(1+\varepsilon f(y))\sin^2\left(\frac{x-y}{2}\right)}dy\\
		& \ \ \ -\frac{1}{\varepsilon}\cdot\frac{(1+\varepsilon g(x))}{\varepsilon^2f'(x)^2+(1+\varepsilon f(x))^2}\int\!\!\!\!\!\!\!\!\!\; {}-{} \frac{\varepsilon^2 f'(x)(f(x)-f(y))\sin(x-y)(1+\varepsilon g(y))}{\varepsilon^2\left(f(x)-f(y)\right)^2+4(1+\varepsilon f(x))(1+\varepsilon f(y))\sin^2\left(\frac{x-y}{2}\right)}dy\\
		& \ \ \ -\frac{1}{\varepsilon}\cdot\frac{(1+\varepsilon g(x))}{\varepsilon^2f'(x)^2+(1+\varepsilon f(x))^2}\int\!\!\!\!\!\!\!\!\!\; {}-{} \frac{\varepsilon(1+\varepsilon f(x))(f(x)-f(y))(1+\varepsilon g(y))}{\varepsilon^2\left(f(x)-f(y)\right)^2+4(1+\varepsilon f(x))(1+\varepsilon f(y))\sin^2\left(\frac{x-y}{2}\right)}dy\\
		& \ \ \ +\frac{1}{\varepsilon}\cdot\frac{(1+\varepsilon g(x))}{\varepsilon^2f'(x)^2+(1+\varepsilon f(x))^2}\int\!\!\!\!\!\!\!\!\!\; {}-{} \frac{\varepsilon f'(x)(1+\varepsilon f(x))\sin(x-y)(1+\varepsilon g(y))}{\varepsilon^2\left(f(x)-f(y)\right)^2+4(1+\varepsilon f(x))(1+\varepsilon f(y))\sin^2\left(\frac{x-y}{2}\right)}dy,\\
	\end{split}
\end{equation*}

\begin{equation*}
	\begin{split}
		&\tilde G_{23}=\frac{(1+\varepsilon g(x))}{\varepsilon^2f'(x)^2+(1+\varepsilon f(x))^2}\int\!\!\!\!\!\!\!\!\!\; {}-{} \frac{(1+\varepsilon f(x))(1+\varepsilon f(y))\cos(x-y)(1+\varepsilon g(y))dy}{\left|(\varepsilon R(x)\cos(x)+\varepsilon R(y)\cos(y)-2d)^2+(\varepsilon R(x)\sin(x)+\varepsilon R(y)\sin(y))^2\right|^2}\\
		& +\frac{(1+\varepsilon g(x))}{\varepsilon^2f'(x)^2+(1+\varepsilon f(x))^2}\int\!\!\!\!\!\!\!\!\!\; {}-{} \frac{\varepsilon f'(x)(1+\varepsilon f(y))\sin(x-y)(1+\varepsilon g(y))dy}{\left|(\varepsilon R(x)\cos(x)+\varepsilon R(y)\cos(y)-2d)^2+(\varepsilon R(x)\sin(x)+\varepsilon R(y)\sin(y))^2\right|^2}\\
		& +\frac{(1+\varepsilon g(x))}{\varepsilon^2f'(x)^2+(1+\varepsilon f(x))^2}\int\!\!\!\!\!\!\!\!\!\; {}-{} \frac{(1+\varepsilon f(x))^2(1+\varepsilon g(y))dy}{\left|(\varepsilon R(x)\cos(x)+\varepsilon R(y)\cos(y)-2d)^2+(\varepsilon R(x)\sin(x)+\varepsilon R(y)\sin(y))^2\right|^2},
	\end{split}
\end{equation*}

\begin{equation*}
	\begin{split}
		&\tilde G_{24}=\frac{(1+\varepsilon g(x))}{\varepsilon^2f'(x)^2+(1+\varepsilon f(x))^2}\int\!\!\!\!\!\!\!\!\!\; {}-{} \frac{2d(1+\varepsilon f(x))\cos(x)(1+\varepsilon g(y))dy}{\left|(\varepsilon R(x)\cos(x)+\varepsilon R(y)\cos(y)-2d)^2+(\varepsilon R(x)\sin(x)+\varepsilon R(y)\sin(y))^2\right|^2}\\
		& +\frac{(1+\varepsilon g(x))}{\varepsilon^2f'(x)^2+(1+\varepsilon f(x))^2}\int\!\!\!\!\!\!\!\!\!\; {}-{} \frac{2\varepsilon df'(x)\sin(x)(1+\varepsilon g(y))dy}{\left|(\varepsilon R(x)\cos(x)+\varepsilon R(y)\cos(y)-2d)^2+(\varepsilon R(x)\sin(x)+\varepsilon R(y)\sin(y))^2\right|^2}.
	\end{split}
\end{equation*}

{\bf Proof of Theorem \ref{thm2}:}
Comparing $G_1,G_2$ with $F_1,F_2$ for co-rotating sheet solutions, we observe that $G_{12}$, $\tilde G_{22}$ equal $F_{12}$, $\tilde F_{22}$ given in \eqref{2-3} and \eqref{2-7} separately. Moreover, for $i=3,4$, $G_{1i}$, $\tilde G_{2i}$ equal $-F_{1i}$, $-\tilde F_{2i}$ in the case $m=2$. Hence by a same procedure as seeking for the regularity of $F_1$, $F_2$, we can prove the continuity of $\mathbf G=(G_1, G_2)$ and their Gateaux derivatives $D\mathbf G$ near $(0,W,0,0)$. From Lemma \ref{lem2-3}, we can obtain the exact matrix form for $D\mathbf G(0,W,0,0)$, which is an isomorphism from $\mathcal A^k$ to $\mathcal B^k$.

 It remains to choose $W$ appropriately such that for $(\varepsilon,W,f,g)\in [-\varepsilon_0,\varepsilon_0]\times \mathbb R \times V^r$ with $\varepsilon_0,r>0$ small enough, the range of $\mathbf G=(G_1, G_2)$ is in $\mathcal B^k$. This can be achieved by letting
\begin{equation*}
	W=W^*+\varepsilon\cdot W_\varepsilon(\varepsilon,f,g)
\end{equation*}
with $W^*=\frac{1}{2d}$ and some regular perturbation $W_\varepsilon(\varepsilon,f,g)$. Then we can apply implicit function theorem on $\mathbf{G^*}(\varepsilon,f,g)$ at $(0,0,0)$ to derive the existence and properties for a local curve of nontrivial traveling vortex sheet solutions.
\qed

\phantom{s}
\thispagestyle{empty}

\end{document}